\documentclass[11pt,a4paper]{amsart}
\usepackage{amsmath,amssymb, amsbsy}
\usepackage{color,psfrag}
\usepackage[dvips]{graphicx}
\usepackage{enumerate}
\newcommand{\R}{{\mathbb R}}
\newcommand{\N}{{\mathbb N}}

\renewcommand{\geq }{\geqslant}

\renewcommand{\leq }{\leqslant}
\def\neweq#1{\begin{equation}\label{#1}}
\def\endeq{\end{equation}}
\def\eq#1{(\ref{#1})}
\newtheorem{theorem}{Theorem}[section]
\newtheorem{proposition}[theorem]{Proposition}
\newtheorem{lemma}[theorem]{Lemma}

\newtheorem{remark}[theorem]{Remark}

\theoremstyle{definition}

\newcommand{\nero }{\color{black}}

\begin{document}

\title[Positivity preserving property for a partially hinged plate]{A positivity preserving property result for the biharmonic operator under partially hinged boundary conditions}

\author[Elvise BERCHIO]{Elvise BERCHIO}
\address{\hbox{\parbox{5.7in}{\medskip\noindent{Dipartimento di Scienze Matematiche, \\
Politecnico di Torino,\\ Corso Duca degli Abruzzi 24, 10129 Torino, Italy. \\[3pt]
\em{E-mail address: }{\tt elvise.berchio@polito.it}}}}}
\author[Alessio FALOCCHI]{Alessio FALOCCHI}
\address{\hbox{\parbox{5.7in}{\medskip\noindent{Dipartimento di Scienze Matematiche, \\
Politecnico di Torino,\\ Corso Duca degli Abruzzi 24, 10129 Torino, Italy. \\[3pt]
\em{E-mail address: }{\tt alessio.falocchi@polito.it}}}}}

\date{\today}



\begin{abstract}

It is well known that for higher order elliptic equations the positivity preserving property (PPP) may fail. In striking contrast to what happens under Dirichlet boundary conditions, we prove that the PPP holds for the biharmonic operator on rectangular domains under partially hinged boundary conditions, i.e. nonnegative loads yield positive solutions. The result follows by fine estimates of the Fourier expansion of the corresponding Green function. 
\end{abstract}

\maketitle

\section{Introduction}
One of the main obstructions in the development of the theory of higher order elliptic equations is represented by the loss of general maximum principles, see e.g. \cite[Chapter 1]{book}. Nevertheless, due to the central role that these technical tolls play in the general theory of second order elliptic equations, in the last century a large part of literature has focused in studying whether the related boundary value problems possibly enjoy the so-called \emph{positivity preserving property}  (PPP in the following). As a matter of example, let us consider the \emph{clamped} plate problem:
\begin{equation}\label{Dirichlet}
\begin{cases}
\Delta^2 u=f \qquad  &\text{in } \Omega \\
u=|\nabla u|=0  \qquad  &\text{on } \partial\Omega
\end{cases}
\end{equation}
where $\Omega \subset \R^n$ is a bounded domain and $f\in L^2(\Omega)$; we say that the above problem satisfies the PPP if the following implication holds
$$f\geq 0 \text{ in } \Omega\quad  \Rightarrow \quad u\geq 0  \text{ in } \Omega\,,$$
where $u$ is a (weak) solution to  \eq{Dirichlet}. The validity of the PPP generally depends either on the choice of the boundary conditions and on the geometry of the domain. For instance, from the seminal works by Boggio \cite{bo1,bo2}, it is known that problem \eq{Dirichlet} satisfies the PPP when $\Omega$ is a ball in $\R^n$, while, in \cite{coffman}, Coffman and Duffin proved that the PPP does not hold when $\Omega$ is a two dimensional domain containing a right angle, such as a square or a rectangle, see also Figure \ref{clamped} below. \par

Things become somehow simpler if in \eq{Dirichlet}, instead of the Dirichlet boundary conditions, we take the Navier boundary conditions, i.e. we consider the \emph{hinged} plate problem:
\begin{equation*}
\begin{cases}
\Delta^2 u=f \qquad  &\text{in } \Omega \\
u=\Delta u=0  \qquad  &\text{on } \partial\Omega.
\end{cases}
\end{equation*}

Here, the PPP follows by applying twice the comparison principle for the laplacian under Dirichlet boundary conditions. It is worth noticing that smoothness of the domain cannot be overlooked since it has been shown by Nazarov and Sweers
\cite{nazarov} that, also in this case, the PPP may fail for planar domains with an interior corner. We refer to the book \cite{book} for more details and PPP results under different kind of boundary conditions, e.g. Steklov boundary conditions, and to \cite{gru,grusweers1,grusweers,parini,romani,schnieders,sweers, sweers3} for up to date results on
the topic. \par
In the present paper we focus on the less studied \emph{partially hinged} plate problem which arises in several mathematical models having engineering interest, e.g. models of bridges or footbridges. In particular, a 2-d model for suspension bridges has been proposed in \cite{fergaz}; here the bridge is seen as a thin long rectangular plate $\Omega \subset \R^2$ hinged at the short edges, see also \cite{bfg} for further details. More precisely, if, by scaling, we assume that $ \Omega=(0,\pi)\times(-\ell,\ell)$ with $\ell>0$, the \emph{partially hinged} problem writes:
	\begin{equation}\label{ppp f(x,y)}
	\begin{cases}
	\Delta^2 u=f \qquad & \text{in } \Omega\,\\
	u(0,y)=u_{xx}(0,y)=u(\pi,y)=u_{xx}(\pi,y)=0  \qquad &\text{for } y\in (-\ell,\ell) \\
	u_{yy}(x,\pm\ell)+\sigma
	u_{xx}(x,\pm\ell)=u_{yyy}(x,\pm\ell)+(2-\sigma)u_{xxy}(x,\pm\ell)=0
	 \qquad& \text{for } x\in (0,\pi),
	\end{cases}
	\end{equation}
	where $f\in L^2(\Omega)$, $\sigma\in [0,1)$ is the so-called Poisson ratio and depends on the material by which the plate is made of. 	\par
	It is known that the validity of the PPP for a problem is related to the sign of the associated Green function. Indeed, if $G_p(q):=G(p,q)$ denotes the Green function of \eqref{ppp f(x,y)}, the (weak) solution to \eqref{ppp f(x,y)} writes
	$$u(p)= \int_{\Omega} G_p(q) f(q)\,dq \quad \forall p\in\Omega$$
	and the PPP becomes equivalent to
	$$G_p(q) \geq 0 \quad \forall (p,q) \in \Omega\,.$$
	The proof of the above inequality represents the main result of the present paper. More precisely, we first write the Fourier expansion of $G_p$, i.e. 
	\begin{equation*}
	G_p(q)=\sum_{m=1}^{+\infty} \dfrac{1}{2\pi}\dfrac{\phi_m(y,w)}{m^3}\sin(m\rho)\,\sin(mx) \qquad \forall p=(\rho,w) \text{ and } q=(x,y)\in\overline \Omega\,,
	\end{equation*}
	where the (involved) analytic expression of the functions $\phi_m$ is given explicitly in formula \eqref{soluzdirac2} of Section \ref{subsecgreen}. As subsequent step, we develop an accurate analysis of the qualitative properties of the $\phi_m$  and we show, in particular, that they are strictly decreasing with respect to $m\in\mathbb{N^+}$. This monotonicity issue is achieved by studying the sign of the derivatives of the $\phi_m$; since they have highly involved analytic expressions, in order to detect their sign, we set up a clever scheme where, step by step, we cancel out the dependence of some variables through optimisation arguments, see Remark \ref{proof idea} of Section \ref{proof1}. From the monotonicity of the $\phi_m$, through an asymptotic analysis, we also deduce their positivity. These information are essential for the subsequent part of the proof where we study the sign of $G_p$. More precisely, by means of suitable lower bounds, we first show the positivity of $G_p$ in a rectangle contained in $\Omega$, far from the hinged edges; then, we obtain the positivity in the remaining parts through suitable iterative procedures which, step by step, stick rectangles where $G_p$ is positive up to the boundary, see Section  \ref{proof2} for all details. \par
	 
As already remarked, the validity of the PPP for problem \eqref{ppp f(x,y)} is by no means an obvious fact, recall that it does not hold for problem \eq{Dirichlet} on rectangular planar domains; furthermore, in general, its validity is not expected for plates having two free edges. In Figure \ref{clamped} (left) we show a well known example of PPP violated for \eqref{Dirichlet} with $\Omega=(0,\pi)\times(-\pi/6,\pi/6)$ and with a load $f$ concentrated in $(\pi/3,0)$, see also \cite{sweers2}. In Figure \ref{clamped} (right) we consider the solution to a \emph{partially clamped} plate problem, i.e. \eqref{ppp f(x,y)} with Dirichlet conditions instead of Navier, with a concentrated load in $(\pi/3,\pi/6)$. Numerically, we obtain regions where the PPP fails near the corners.
	\begin{figure}[!hbt]
		\centering
		{\includegraphics[width=15cm]{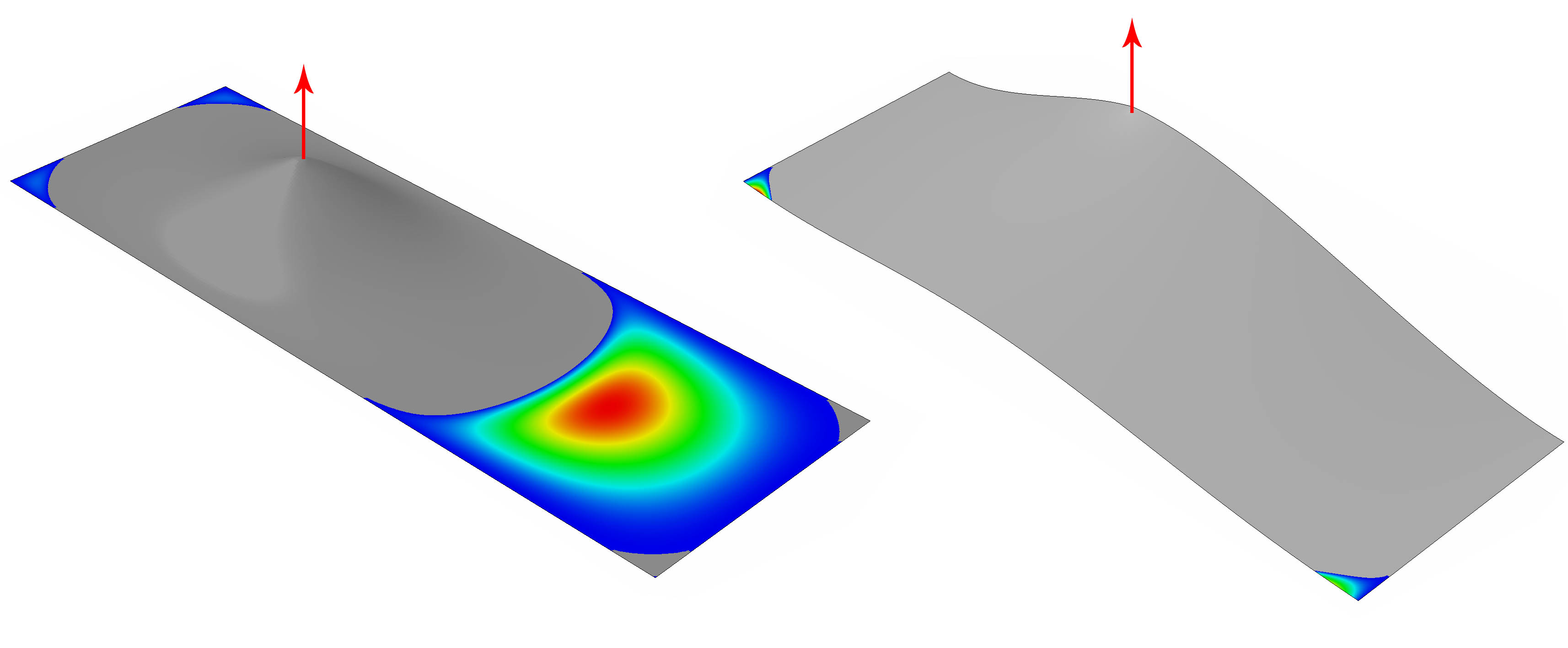}}
		\caption{Finite element approximated solution $u$ of the clamped (left) and partially clamped (right) plate problems under, respectively, a concentrated load in $(\frac{\pi}{3},0)$ and in $(\frac{\pi}{3},\frac{\pi}{6})$; the regions where $u\geq 0$ are grey, while the regions where $u<0$ are coloured from blue (less negative) to red ($\Omega=(0,\pi)\times(-\pi/6,\pi/6)$, $\sigma=0.2$).}
		\label{clamped}
	\end{figure}

	 The paper is organized as follows. In Section \ref{prelim} we introduce some notations and we state our main results: the Fourier expansion of $G_p$, together with the qualitative properties of its components, which is given in Theorem \ref{monotonia} and the precise statement of the PPP result which is given in Theorem \ref{greenpositiva}. The rest of the paper is devoted to the proofs. More precisely, in Section \ref{subsecgreen} we compute explicitly the Fourier series of the Green function as the limit of the solution to \eqref{ppp f(x,y)} for a specific $L^2$ forcing term converging to the Dirac delta function. In Section \ref{proof1} we prove the monotonicity and the positivity of the $\phi_m$, while in Section \ref{proof2}  we show the positivity of the Green function. Finally, we collect in the Appendix the proofs of some technical results needed either in Section \ref{proof1} and in Section \ref{proof2}.
	\par

	\section{Notations and main results}\label{prelim}
	
	The natural functional space where to set problem \eq{ppp f(x,y)} is
	$$
	H^2_*(\Omega)=\big\{u\in H^2(\Omega): u=0\mathrm{\ on\ }\{0,\pi\}\times(-\ell,\ell)\big\}\,.
	$$
	Note that the condition $u=0$ has to be meant in a classical sense because $\Omega$ is a planar domain and the energy space $H^2_*(\Omega)$ embeds into continuous functions. 
	Furthermore, for $\sigma\in[0,1)$ fixed, $H^2_*(\Omega)$ is a Hilbert space when endowed with the scalar product
	$$
	(u,v)_{H^2_*(\Omega)}:=\int_\Omega \left[\Delta u\Delta v+(1-\sigma)(2u_{xy}v_{xy}-u_{xx}v_{yy}-u_{yy}v_{xx})\right]\, dx \, dy \,
	$$
	with associated norm
	$$
	\|u\|_{H^2_*(\Omega)}^2=(u,u)_{H^2_*(\Omega)} \, ,
	$$
	which is equivalent to the usual norm in $H^2(\Omega)$, see \cite[Lemma 4.1]{fergaz}. Then, we reformulate problem \eq{ppp f(x,y)} in the following weak sense
	\begin{equation}
	\label{eigenweak}
	(u,v)_{H^2_*(\Omega)} =(f,v)_{L^2(\Omega)} \qquad\forall v\in H^2_*(\Omega).
	\end{equation}
	If $f\in H^{-2}_*(\Omega)$ we write $\langle f,v\rangle$ instead of $(f,v)_{L^2(\Omega)}$, i.e.
	\begin{equation}
	\label{weakdual}
	(u,v)_{H^2_*(\Omega)} =\langle f,v\rangle \qquad\forall v\in H^2_*(\Omega).
	\end{equation}
	Clearly, problem \eq{weakdual} (and consequently \eq{eigenweak}) admits a unique solution $u\in H^2_*(\Omega)$; in the following we shall specify the cases when $f\in H^{-2}_*(\Omega)$, otherwise we will always assume  $f\in L^2(\Omega)$.\par 
	For all $p\in \Omega$, the \emph{Green function} $G_p$ of \eqref{ppp f(x,y)} is, by definition, the unique solution to 
	\begin{equation}
	\label{greeneq}
	(G_p,v)_{H^2_*(\Omega)} =\langle \delta_{p},v\rangle=v(p) \qquad\forall v\in H^2_*(\Omega)\,.
	\end{equation}
	Recalling that $H^2_*(\Omega) \subset C^0(\overline \Omega)$ the above definition makes sense for all $p \in \overline \Omega$.
	By separating variables, in Section \ref{subsecgreen} we derive the Fourier expansion of $G_p$ and in Section \ref{proof1} we prove some crucial qualitative properties of its Fourier components. We collect these results in the following:
	
	\begin{theorem}\label{monotonia}
	Let $\sigma\in[0,1)$ and $p=(\rho,w)\in \overline \Omega$, furthermore let $G_p\in H^2_*(\Omega) \subset C^0(\overline \Omega)$ be the Green function of \eqref{ppp f(x,y)}. Then,
	\begin{equation*}
	G_p(x,y)=\sum_{m=1}^{+\infty} \dfrac{1}{2\pi}\dfrac{\phi_m(y,w)}{m^3}\sin(m\rho)\,\sin(mx) \qquad \forall (x,y)\in\overline \Omega\,,
	\end{equation*}
where the functions $\phi_m(y,w)$ are given explicitly in formula \eqref{soluzdirac2} of Section \ref{subsecgreen}.\par
	 In particular, the $\phi_m(y,w)$ are strictly positive and strictly decreasing with respect to $m$, i.e.
			\begin{equation}\label{decreasing}
		0<\phi_{m+1}(y,w)<\phi_m(y,w)\qquad\forall m\in\mathbb{N^+},\forall y,w\in[-\ell,\ell]\,.
		\end{equation}

	\end{theorem}
	\par
	\medskip
	\par
	
		\begin{figure}[!hbt]
		\centering
		{\includegraphics[width=8cm]{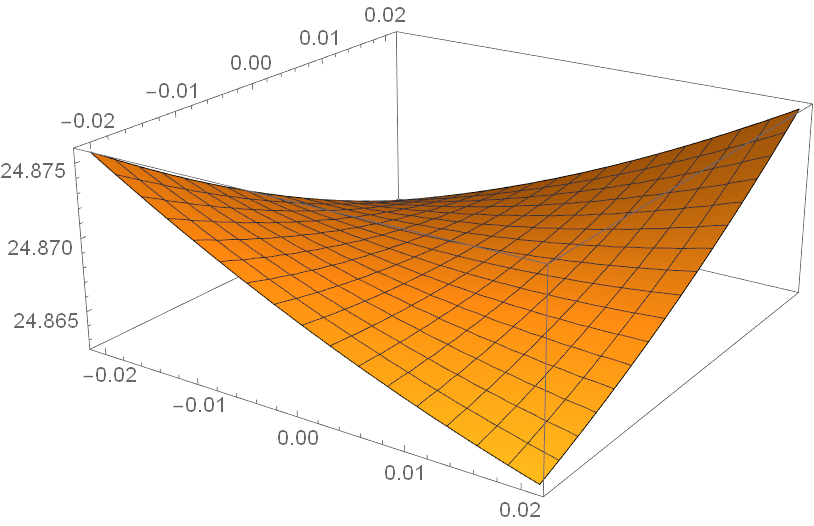}\hspace{5mm}\includegraphics[width=8cm]{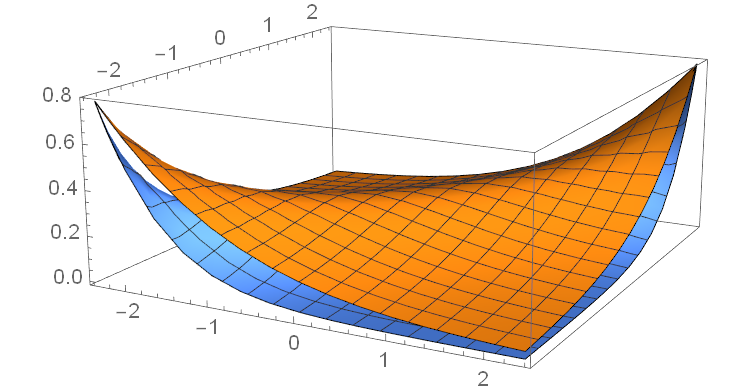}}
		\caption{On the left plot of $\phi_1(y,w)$ with $\ell=\pi/150$ and $\sigma=0.2$; on the right plot of $\phi_1(y,w)$ (orange) and $\phi_2(y,w)$ (blue) with $\ell=3\pi/4$ and $\sigma=0.2$.}
		\label{greenfig}
	\end{figure}
	In Figure \ref{greenfig} on the left we provide the plot of $\phi_1(y,w)$ with $\ell=\pi/150$ and $\sigma=0.2$; on the right we provide the plot of $\phi_1(y,w)$ and $\phi_2(y,w)$ for $\ell=3\pi/4$. Qualitatively, we have similar plots for any $m\in\mathbb{N^+}$ and they all highlight that the points where the positivity of $\phi_m(y,w)$ is more difficult to show are $(\pm\ell,\mp\ell)$.  This confirms the physical intuition that a concentrated load in $w=\ell$ produces the largest vertical (positive) displacement in $y=\ell$ and the smallest in $y=-\ell$. We refer to \cite{bbgza} for a detailed analysis about the torsional performances of partially hinged plates under the action of different external forces.  \par 
	Instead, Figure \ref{greenfig} on the right highlights how the monotonicity issue (with respect $m$) becomes more difficult to be proved at $(\pm\ell,\pm\ell)$, where the difference between the $\phi_m$ reduces. Numerically, we see that this becomes more evident for large $\ell$. However, Theorem \ref{monotonia} assures that the $\phi_m$ never intersect and preserve their positivity for all $\ell>0$.  
	
	By exploiting Theorem \ref{monotonia} we derive the main result of the paper, namely the positivity of $G_p$. More precisely, we set
 $$\widetilde{\Omega}:=(0,\pi)\times[-\ell,\ell]\,$$ 
 and we prove
	
	\begin{theorem}\label{greenpositiva}
	Let $\sigma\in[0,1)$ and $p\in \widetilde \Omega$, furthermore let $G_p\in H^2_*(\Omega) \subset C^0(\overline \Omega)$  be the Green function of \eqref{ppp f(x,y)}. There holds
		\begin{equation*}\label{tesi}
		G_p(x,y)> 0\qquad \forall (x,y) \in\widetilde \Omega.
		\end{equation*}
		Therefore, if $f\in L^2(\Omega)$ and $u$ is the solution of \eq{ppp f(x,y)}, the following implication holds
		\begin{equation*}\label{tesi0}
		f\geq 0,\,\, f\not\equiv 0  \text{ in } \Omega\quad  \Rightarrow \quad u>0  \text{ in } \widetilde{\Omega}.
		\end{equation*}
	\end{theorem}
	As explained in the Introduction, the validity of the above implication is not obvious at all; recall that the positivity issue fails on rectangular plates under Dirichlet boundary conditions, see \cite{coffman} and Figure \ref{clamped}.

 \begin{remark}
	The Poisson ratio $\sigma$ of a material is defined as the ratio between the transversal strain and the longitudinal strain in the direction of the stretching force; for most of materials we have $\sigma\in(0,1/2)$. Nevertheless, there are materials having negative Poisson ratio, hence the range $\sigma\in(-1,1/2)$ includes all possible values. Numerical experiments lead us to conjecture that Theorem \ref{greenpositiva} still holds for $\sigma\in (-1,0)$. In Remark \ref{knegative} of Section \ref{proof1} we highlight the points where our proof fails when assuming $\sigma$ negative.  
\end{remark}

	\section{Green function computation}\label{subsecgreen}
	
	The aim of this section is to provide the Fourier expansion of the Green function $G_p$, namely of the solution to \eqref{greeneq}. This is done by developing a suitable limit approach where, in principle, $ \delta_{p}$ is replaced by a suitable $L^2$ function converging to it. \par
	 To begin with, we fix $p=(\rho,w)\in \Omega$ and we introduce $\alpha,\eta>0$ sufficiently small so that $[\rho-\alpha,\rho+\alpha]\times[w-\eta,w+\eta]\subset  \Omega$; then we denote by $u^p_{\alpha,\eta} \in H^2_*(\Omega)$ the unique solution to the auxiliary problem:
	\begin{equation}\label{ppp approx}
	(u^p_{\alpha,\eta},v)_{H^2_*(\Omega)} =(f^p_{\alpha,\eta},v)_{L^2(\Omega)} \qquad\forall v\in H^2_*(\Omega),
	\end{equation}
	where $$f^p_{\alpha,\eta}(x,y):=\dfrac{\chi_{[\rho-\alpha,\rho+\alpha]}(x)\chi_{[w-\eta,w+\eta]}(y)}{4\alpha\eta}$$
	and $\chi_A$ denotes the characteristic function of the set $A\subset \mathbb{R}$. We get:
	\begin{lemma}\label{conv green}
		As $(\alpha,\eta)\rightarrow(0,0)$ there holds
		$$
		f^p_{\alpha,\eta}\rightarrow \delta_p \quad \text{in } H^{-2}_*(\Omega)\label{statement 1} \quad\text{ and } \quad
		u^p_{\alpha,\eta}\rightarrow G_p\quad\text{in }H^2_*(\Omega) \text{ and in } C^0(\overline \Omega) $$
		where $G_p$ is the unique solution to \eqref{greeneq}.
	\end{lemma}
		\begin{proof}
		We start by showing that $f^p_{\alpha,\eta}\rightarrow \delta_p$ in $H^{-2}_*(\Omega)$, i.e. that
		$$
		\lim\limits_{(\alpha,\eta)\rightarrow(0,0)}\int_\Omega f^p_{\alpha,\eta}(x,y)v(x,y)\,dx\,dy=v(p)=v(\rho,w)\qquad\forall v\in H^2_*(\Omega).
		$$
		Since $v\in H^2_*(\Omega)\subset C^0(\overline \Omega)$ by the mean value theorem there exists $(\sigma_{\alpha,\eta},\tau_{\alpha,\eta})\in [\rho-\alpha,\rho+\alpha]\times[w-\eta,w+\eta]$ such that
		$$
		\int_\Omega f^p_{\alpha,\eta}(x,y)v(x,y)\,dx\,dy=v(\sigma_{\alpha,\eta},\tau_{\alpha,\eta})\rightarrow v(\rho,w) \quad\text{as }(\alpha,\eta)\rightarrow(0,0)\qquad\forall v\in H^2_*(\Omega).
		$$
		Finally, we subtract  \eqref{greeneq} from \eqref{ppp approx} and testing with $v=u^p_{\alpha,\eta}-G_p$ we obtain 
		$$\|u^p_{\alpha,\eta}-G_p\|_{H^2_*(\Omega)}\leq \|f^p_{\alpha,\eta}-\delta_p\|_{H^{-2}_*(\Omega)}\rightarrow0\qquad\text{as }(\alpha,\eta)\rightarrow(0,0),$$
		implying that $u^p_{\alpha,\eta}\rightarrow G_p$ in $H^2_*(\Omega)$. 	
	\end{proof}

	Next we provide the explicit Fourier expansion of $u^p_{\alpha,\eta}$. To this aim we set:
		\begin{equation}\label{ci}
	\begin{split}
	c_1:=&c_1(m,w,\eta)=\frac{m A_m(\ell)[V_{m,w,\eta}(\ell)+V_{m,w,\eta}(-\ell)]+B_m(\ell)[W_{m,w,\eta}(\ell)-W_{m,w,\eta}(-\ell)]}{2m^3(1-\sigma) F(m\ell)}\\
	c_2:=&c_2(m,w,\eta)=\frac{m \overline A_m(\ell)[V_{m,w,\eta}(\ell)-V_{m,w,\eta}(-\ell)]+\overline B_m(\ell)[W_{m,w,\eta}(\ell)+W_{m,w,\eta}(-\ell)]}{2m^3(1-\sigma) \overline F(m\ell)}\\
	c_3:=&c_3(m,w,\eta)=\frac{m \cosh(m\ell)[V_{m,w,\eta}(\ell)-V_{m,w,\eta}(-\ell)]-\sinh(m\ell)[W_{m,w,\eta}(\ell)+W_{m,w,\eta}(-\ell)]}{2m^2 \overline F(m\ell)}\\
	c_4:=&c_4(m,w,\eta)=\frac{m \sinh(m\ell)[V_{m,w,\eta}(\ell)+V_{m,w,\eta}(-\ell)]-\cosh(m\ell)[W_{m,w,\eta}(\ell)-W_{m,w,\eta}(-\ell)]}{2m^2 F(m\ell)}
	\end{split}
	\end{equation}
	where
	$$F(m\ell):=\dfrac{(3+\sigma)}{2}\sinh(2m\ell)-m\ell (1-\sigma)\,, \quad  \overline F(m\ell):=\dfrac{(3+\sigma)}{2}\sinh(2m\ell)+m\ell (1-\sigma)$$
	$$A_m(\ell):=(1+\sigma)\sinh(m\ell)-(1-\sigma)m\ell \cosh(m\ell)\,, \quad B_m(\ell):= 2\cosh(m\ell)+(1-\sigma)m\ell\sinh(m\ell)\,,$$
	$$\overline A_m(\ell):=(1+\sigma)\cosh(m\ell)-(1-\sigma)m\ell \sinh(m\ell)\,,\quad \overline B_m(\ell):= 2\sinh(m\ell)+(1-\sigma)m\ell\cosh(m\ell)\,,$$
	$$V_{m,w,\eta}(\ell):=\sigma m^2 \Phi_{m,w,\eta}(\ell)-(\Phi_{m,w,\eta})''(\ell)\,,$$
	$$W_{m,w,\eta}(\ell):=(\sigma-2) m^2 (\Phi_{m,w,\eta})'(\ell)+(\Phi_{m,w,\eta})'''(\ell)\,,$$
	$$V_{m,w,\eta}(-\ell):=\sigma m^2 \Phi_{m,w,\eta}(-\ell)-(\Phi_{m,w,\eta})''(-\ell)\,,$$ 
	$$W_{m,w,\eta}(-\ell):=(\sigma-2) m^2 (\Phi_{m,w,\eta})'(-\ell)+(\Phi_{m,w,\eta})'''(-\ell)\,,$$
	with
	\begin{equation}\label{Phimweta}
	\Phi_{m,w,\eta}(y):=\frac{1}{2\eta}\int_{w-\eta}^{w+\eta}\,\frac{(1+m|y-t|)e^{-m|y-t|}}{4m^3}\,dt\,.
		\end{equation}
We notice that $\Phi_{m,w,\eta}$ is given by the convolution of the $H^3(\R)$ function $\frac{(1+m|y|)e^{-m|y|}}{4m^3}$ and the $L^2(\R)$ function $\frac{\chi_{[w-\eta,w+\eta]}(y)}{2\eta}$, hence $\Phi_{m,w,\eta}\in C^3(\R)$ and all the above constants are well-defined. Then we prove

	\begin{lemma}\label{convfou}
		Let  $u^p_{\alpha,\eta}$ be the unique solution to \eqref{ppp approx}, then
$$u^p_{\alpha,\eta}(x,y)=\sum_{m=1}^{+\infty} \varphi_{m,\alpha,\eta}^p(y)\sin(mx)$$
with
\begin{equation} \label{varphi_coeff}
	\begin{split}
	\varphi_{m,\alpha,\eta}^p(y):=&{\dfrac{2}{\pi}}\dfrac{\sin(m\alpha)}{ m\alpha \,}\sin(m\rho) \cdot\\&\cdot\left[ c_1 \cosh(my)+c_2 \sinh(my)+c_3 y\cosh(my) +c_4y\sinh(my)+\Phi_{m,w,\eta}(y) \right]
	\end{split}
	\end{equation}
where the constants $c_i$ and $\Phi_{m,w,\eta}$ are defined in \eqref{ci} and \eqref{Phimweta}. Furthermore, the above series converges in $H^2_*(\Omega)$ and in $C^0(\overline \Omega)$.
	\end{lemma}
	
	\begin{proof}
	First we set
	$$f_{m,\alpha,\eta}^p(y):={\dfrac{2}{\pi}}\int_{0}^{\pi} f^p_{\alpha,\eta}(x,y)\, \sin(mx)\, dx=\dfrac{\chi_{[w-\eta,w+\eta]}(y)}{\pi\eta}\dfrac{\sin(m\alpha)}{m\alpha}\sin(m\rho)\,.$$
	Then, for $M\geq 1$ we define
	\begin{equation*}\label{fouapprox}
	u^{p, M}_{\alpha,\eta}(x,y)(x,y)=\sum_{m=1}^M \varphi_{m,\alpha,\eta}^p(y) \sin(mx)\,
	\end{equation*}
	where, for each $1\leq m\leq M$, $\varphi_m=\varphi_{m,\alpha,\eta}^p(y)$ is the unique solution to the problem:
	\begin{equation}\label{eqppp}
	a_m(\varphi_m,\phi)=(f_{m,\alpha,\eta}^p,\phi)_{L^2(-\ell,\ell)} \qquad \forall\phi\in H^2(-\ell, \ell)\,,
	\end{equation}
	with 
	$$
	a_m(\varphi,\phi):=\int_{-\ell}^\ell [\varphi''\phi''+2m^2(1-\sigma)\varphi'\phi'-\sigma m^2(\varphi''\phi+\varphi\phi'')+m^4\varphi\phi]\,dy
	$$
	continuous and coercive bilinear form in $H^2(-\ell, \ell)$ with associated norm
	$$\|\varphi\|^2_{H^2_m(-\ell, \ell)}:=a_m(\varphi,\varphi)\,.$$
In strong form problem \eqref{eqppp} reads
	\begin{equation*}\label{varphi}
	\begin{cases}
	\varphi_m''''(y)-2m^2\varphi_m''(y)+m^4\varphi_m(y)=f_{m,\alpha,\eta}^p(y)&y\in(-\ell,\ell)\\
	\varphi_m''(\pm\ell)-\sigma m^2\varphi_m(\pm\ell)=0 & \qquad  \\
	\varphi_m'''(\pm\ell)-(2-\sigma)m^2\varphi_m'(\pm\ell)=0\,& \qquad 
	\end{cases}
	\end{equation*}
 and some computations yield that the $\varphi_{m,\alpha,\eta}^p$ are as given in \eqref{varphi_coeff}, see \cite[Theorem 5.1]{befafega} for details. 
 
 The proof of Lemma \ref{convfou} follows by showing that
		$$u^{p,M}_{\alpha,\eta}(x,y)\rightarrow u^p_{\alpha,\eta}(x,y) \quad \text{in }H^2_*(\Omega) \quad \text{as } M\rightarrow+\infty$$
		where $u^p_{\alpha,\eta}(x,y)$ is the unique solution of \eqref{ppp approx}. Let $v\in H^2_*(\Omega)$, it is readily checked that, for $M\geq 1$ fixed, $u^{p,M}_{\alpha,\eta}$ satisfies
		\begin{equation}
		\label{eigenweak2}
		(u^{p,M}_{\alpha,\eta},v^M)_{H^2_*(\Omega)} =(f^{p,M}_{\alpha,\eta}, v^M)_{L^2(\Omega)} \,,
		\end{equation}
		where 
		$$v^M(x,y):=\sum_{m=1}^M v_m(y)\sin(mx) \quad \text{with }  v_m(y):= \frac{2}{\pi} \int_{0}^{\pi} v(x,y)\, \sin(mx)\, dx$$
		and
		$$f^{p,M}_{\alpha,\eta}(x,y)=\sum_{m=1}^M f_{m,\alpha,\eta}^p(y)\sin(mx) \,.$$   
		Since $f^{p}_{\alpha,\eta}\in L^2(\Omega)$, a well-known result for Fourier series yields
		$\int_{0}^{\pi} (f^{p}_{\alpha,\eta}(x,y))^2\,dx=\frac{\pi}{2}\sum_{m=1}^{+\infty}( f_{m,\alpha,\eta}^p(y))^2\,.$
		Hence, by direct computation we infer that
		$$\|f^{p,M}_{\alpha,\eta} \|_{L^2(\Omega)}^2=\frac{\pi}{2}\sum_{m=1}^{M} \|f_{m,\alpha,\eta}^p(y)\|_{L^2(-\ell, \ell)}^2\leq \|f^{p}_{\alpha,\eta} \|_{L^2(\Omega)}^2\,.$$
		From the above inequality we deduce that $f^{p,M}_{\alpha,\eta}$ is a Cauchy sequence in $L^2(\Omega)$ and, in turn, that $f^{p,M}_{\alpha,\eta} \rightarrow f^{p}_{\alpha,\eta}$ in $L^2(\Omega)$ as $M \rightarrow + \infty$.
		Similarly, by the fact that $v\in H^2_*(\Omega)$ we infer that
		$$\int_{0}^{\pi} v_{xx}^2(x,y)\,dx=\frac{\pi}{2}\sum_{m=1}^{+\infty}m^4 v_m^2(y)\,, \quad \int_{0}^{\pi} v_{yy}^2(x,y)\,dx=\frac{\pi}{2}\sum_{m=1}^{+\infty} (v_m''(y))^2$$
		$$\text{and} \quad \int_{0}^{\pi} v_{xy}^2(x,y)\,dx=\frac{\pi}{2}\sum_{m=1}^{+\infty}m^2(v_m'(y))^2\,.$$
		By this, a direct computation yields
		$$\|v^M\|_{H^2_*(\Omega)}^2=\frac{\pi}{2}\sum_{m=1}^{M} \|v_m(y)\|^2_{H^2_m(-\ell, \ell)} \leq \|v\|_{H^2_*(\Omega)}^2 $$
		and, in turn, we conclude that $v^M\rightarrow v \quad \text{in }H^2_*(\Omega)$ as $M \rightarrow + \infty$.
		It remains to prove that $u^{p,M}_{\alpha,\eta}(x,y)\rightarrow u^p_{\alpha,\eta}(x,y) $ in $H^2_*(\Omega)$ for $M\rightarrow+\infty$. Finally, we observe that 
		$$\|u^{p,M}_{\alpha,\eta} \|^2_{H^2_*(\Omega)}=\frac{\pi}{2}\sum_{m=1}^M\|\varphi_{m,\alpha,\eta}^p\|^2_{H^2_m(-\ell,\ell)}=\frac{\pi}{2}\sum_{m=1}^{M} \|f_{m,\alpha,\eta}^p(y)\|_{L^2(-\ell, \ell)}^2 \leq \|f^{p}_{\alpha,\eta}\|_{L^2(\Omega)}^2\,$$
		assuring that $u^{p,M}_{\alpha,\eta}\rightarrow u^p_{\alpha,\eta}$ in $H^2_*(\Omega)$ as $M\rightarrow+\infty$. For what remarked, the proof of the statement follows by passing to the limit in \eqref{eigenweak2}.
	\end{proof}

In order to write explicitly the Fourier series of $G_p$, for all $m\in \N_+$, we set
	 \begin{equation}\label{soluzdirac2}
		\begin{split}
		\phi_m(y,w):=
		&e^{-m\ell}\bigg[\cosh(m w)\bigg(\dfrac{\overline \zeta(my,m\ell)}{F(m\ell)}+m\ell\dfrac{\overline \psi(my,m\ell) }{F(m\ell)}-m w\dfrac{\overline \xi(my,m\ell)}{\overline F(m\ell)}\bigg)\\&+\sinh(m w)\bigg(\dfrac{\overline\eta(my,m\ell)}{\overline F(m\ell)}+m\ell\dfrac{\overline \xi(my,m\ell) }{\overline F(m\ell)}-mw\dfrac{\overline\psi(my,m\ell)}{F(m\ell)}\bigg)\bigg] 
		\\&+(1+m|y-w|)e^{-m|y-w|}
		\end{split}
		\end{equation}
where the functions $F,\overline F: (0,+\infty)\rightarrow \mathbb{R}$ and $\overline\zeta, \overline\eta, \overline\psi, \overline\xi: \mathbb{R}\times(0,+\infty)\rightarrow \mathbb{R}$ are defined as follows
	\begin{equation}\label{cost1}
	\begin{split}
	F(z):=&\dfrac{(3+\sigma)}{2}\sinh(2z)-z (1-\sigma)\,,\\
	\overline F(z):=&\dfrac{(3+\sigma)}{2}\sinh(2z)+z (1-\sigma)\,,\\
	\overline\zeta(r,z):=&\bigg(\dfrac{4}{1-\sigma}-z(1+\sigma)\bigg)\cosh(r)\cosh(z)+\bigg(\dfrac{(1+\sigma)^2}{1-\sigma}+2z\bigg)\cosh(r)\sinh(z)\\&-2r\sinh(r)\cosh(z)+r(1+\sigma)\sinh(r)\sinh(z)\\
	\overline \eta(r,z):=&r(1+\sigma)\cosh(r)\cosh(z)-2r\cosh(r)\sinh(z)\\&+\bigg(\dfrac{(1+\sigma)^2}{1-\sigma}+2z\bigg)\sinh(r)\cosh(z)+\bigg(\dfrac{4}{1-\sigma}-z(1+\sigma)\bigg)\sinh(r)\sinh(z)\\
	\overline\psi(r,z):=&\big(2+(1-\sigma)z\big)\cosh(r)\cosh(z)+\big(-(1+\sigma)+z(1-\sigma)\big)\cosh(r)\sinh(z)\\&-r(1-\sigma)\sinh(r)\cosh(z)-r(1-\sigma)\sinh(r)\sinh(z)\\
	\overline \xi(r,z):=&-r(1-\sigma)\cosh(r)\cosh(z)-r(1-\sigma)\cosh(r)\sinh(z)\\&+\big(-(1+\sigma)+z(1-\sigma)\big)\sinh(r)\cosh(z)+\big(2+(1-\sigma)z\big)\sinh(r)\sinh(z)\,.
	\end{split}
	\end{equation}
	
	From Lemma \ref{convfou} we derive:
	\begin{proposition}\label{green expansion}
	Let $\sigma\in[0,1)$ and $p=(\rho,w)\in \overline \Omega$, furthermore let $G_p\in H^2_*(\Omega)$ be as in \eqref{greeneq}. Then,
\begin{equation}\label{green}
	G_p(x,y)=\sum_{m=1}^{+\infty} \dfrac{1}{2\pi}\dfrac{\phi_m(y,w)}{m^3}\sin(m\rho)\, \sin(mx) \qquad \forall(x,y)\in \overline \Omega\,,
\end{equation}
		where the functions $\phi_m(y,w)$ are given in  \eqref{soluzdirac2}. Moreover, the series in \eqref{green} converges in $H^2_*(\Omega)$ and in $C^0(\overline{\Omega})$.
	\end{proposition}
\begin{proof}
	Let $u^p_{\alpha,\eta}\in H^2_*(\Omega)$ be the unique solution to \eqref{ppp approx}; by expanding in Fourier series, we have that
	$$u^p_{\alpha,\eta}(x,y)=\sum_{m=1}^{+\infty} \varphi_{m,\alpha,\eta}^p(y)\sin(mx) \quad \text{and } \quad G_p(x,y)=\sum_{m=1}^{+\infty} g_m^p(y)\sin(mx)$$
	with
	$$\varphi_{m,\alpha,\eta}^p(y)= \frac{2}{\pi} \int_{0}^{\pi} u^p_{\alpha,\eta}(x,y)\, \sin(mx)\, dx \quad \text{and } \quad g_m^p(y)=  \frac{2}{\pi} \int_{0}^{\pi} G_p(x,y)\, \sin(mx)\, dx\,.$$
	
	Passing to the limit above and thanks to Lemma \ref{conv green}, we infer that 
	$$\varphi_{m,\alpha,\eta}^p(y) \rightarrow g_m^p(y)\quad \text{ in } C^0([-\ell, \ell]) \qquad \text{as } (\alpha,\eta)\rightarrow(0,0)\,.$$
	
	On the other hand, from Lemma \ref{convfou} we know that the $\varphi_{m,\alpha,\eta}^p$ write as in \eqref{varphi_coeff}.
	Now, as $\eta\rightarrow 0$, a direct inspection reveals that:
	$$\Phi_{m,w,\eta}(y) \rightarrow \frac{(1+m|y-w|)e^{-m|y-w|}}{4m^3}:=\overline \Phi_m(y,w)\quad \text{ in } C^0([-\ell, \ell])$$
	$$V_{m,w,\eta}(\ell)\rightarrow \frac{e^{-m(\ell-w)}}{4m}(1+\sigma-m(\ell-w)(1-\sigma))  \,,\quad \quad W_{m,w,\eta}(\ell)\rightarrow \frac{e^{-m(\ell-w)}}{4}(2+m(\ell-w)(1-\sigma)) \,$$
	and
	$$V_{m,w,\eta}(-\ell)\rightarrow  \frac{e^{-m(\ell+w)}}{4m}(1+\sigma-m (\ell+w)(1-\sigma)) \,,\quad \quad W_{m,w,\eta}(-\ell)\rightarrow -\frac{e^{-m(\ell+w)}}{4}(2+m(\ell+w)(1-\sigma))\,.$$
	Therefore,
	\small
	\begin{equation*}
	\begin{split}
	c_1(m,w,\eta)\rightarrow \overline c_1(m,w):=&\dfrac{e^{-m\ell}}{4m^3F(m\ell)}\bigg[\bigg(m^2\ell^2(1-\sigma)+m\ell(1-\sigma)+\dfrac{(1+\sigma)^2}{1-\sigma}\bigg)\sinh(m\ell)\cosh(mw)\\&\bigg(m^2\ell^2(1-\sigma)+m\ell(1-\sigma)+\dfrac{4}{1-\sigma}\bigg)\cosh(m\ell)\cosh(mw)\\&-[2+m\ell(1-\sigma)]mw\cosh(m\ell)\sinh(mw)+[1+\sigma-m\ell(1-\sigma)]mw\sinh(m\ell)\sinh(mw)\bigg]\\c_2(m,w,\eta)\rightarrow \overline c_2(m,w):=&\dfrac{e^{-m\ell}}{4m^3\overline F(m\ell)}\bigg[\bigg(m^2\ell^2(1-\sigma)+m\ell(1-\sigma)+\dfrac{(1+\sigma)^2}{1-\sigma}\bigg)\sinh(mw)\cosh(m\ell)\\&\bigg(m^2\ell^2(1-\sigma)+m\ell(1-\sigma)+\dfrac{4}{1-\sigma}\bigg)\sinh(m\ell)\sinh(mw)\\&-[2+m\ell(1-\sigma)]mw\sinh(m\ell)\cosh(mw)+[1+\sigma-m\ell(1-\sigma)]mw\cosh(m\ell)\cosh(mw)\bigg]\\c_3(m,w,\eta)\rightarrow \overline c_3(m,w):=&\dfrac{e^{-m\ell}}{2m^2\overline F(m\ell)}\bigg[\big(1+\sigma-m\ell(1-\sigma)\big)\sinh(mw)\cosh(m\ell)+(1-\sigma)mw\sinh(m\ell)\cosh(mw)\\&-\big(2+m\ell(1-\sigma)\big)\sinh(m\ell)\sinh(mw)+(1-\sigma)mw\cosh(m\ell)\cosh(mw)\bigg]\\c_4(m,w,\eta)\rightarrow \overline c_4(m,w):=&\dfrac{e^{-m\ell}}{2m^2 F(m\ell)}\bigg[\big(1+\sigma-m\ell(1-\sigma)\big)\sinh(m\ell)\cosh(mw)+(1-\sigma)mw\sinh(mw)\sinh(m\ell)\\&-\big(2+m\ell(1-\sigma)\big)\cosh(m\ell)\cosh(mw)+(1-\sigma)mw\cosh(m\ell)\sinh(mw)\bigg]\,.
	\end{split}
	\end{equation*}
	\normalsize
	The above limits inserted in \eqref{varphi_coeff} yield
	$$\varphi_{m,\alpha,\eta}^p(y) \rightarrow\frac{ \phi_m(y,w)}{2\pi m^3}\sin(m\rho)\qquad \text{as }(\alpha,\eta)\rightarrow (0,0),$$ 	
	where $\phi_m$ is as given in \eqref{soluzdirac2}, which proves \eqref{green} for all $p\in\Omega$. Let now $\overline p \in \overline \Omega$ and let $G_{\overline{p}}$ be the corresponding solution to  \eqref{greeneq}. It is readily seen that $
		\delta_{p_n}\rightarrow\delta_{\overline p}$ in $H^{-2}_*(\Omega)$ for all $\{ p_n\}\subset \Omega: p_n\rightarrow \overline{p}$; then, arguing as in Lemma \ref{conv green}, it follows that $G_{p_{n}}\rightarrow G_{\overline{p}}$ in $H^2_*(\Omega)$ and, consequently, in $C^0(\overline \Omega)$. By this we infer that \eqref{green} extends continuously to all $p\in\overline{\Omega}$.

The convergence of the series \eqref{green} in $H^2_*(\Omega)$ and in $C^0(\overline{\Omega})$ can be easily checked by exploiting the monotonicity property \eqref{decreasing} (see Section \ref{monot} for the proof). Indeed, we have
				$$
		|g_m^p(y)\sin(mx)|\leq \dfrac{1}{2\pi}\dfrac{\phi_m(y,w)}{m^3}|\sin(m\rho)|\leq  \dfrac{1}{2\pi}\dfrac{\|\phi_1\|_{\infty}}{m^3}\leq \dfrac{C}{m^3},
			$$
			by which the convergence in  $C^0(\overline{\Omega})$ follows at once. The convergence in $H^2_*(\Omega)$ follows from similar estimates.
		
\end{proof}

\section{Proof of Theorem \ref{monotonia}}\label{proof1}
The first part of the statement, namely the Fourier expansion of the Green function, has already been derived in the previous section, see Proposition \ref{green expansion}. Here we focus on the sign and monotonicity properties of the functions $\phi_m(y,w)$.

\subsection{Proof of the monotonicity issue in \eqref{decreasing}}\label{monot}
We rewrite the functions $\phi_m(y,w)$ in a more convenient way; to this aim we introduce the functions $\zeta, \eta, \psi, \xi: [-1,1]\times(0,+\infty)\rightarrow \mathbb{R}$:
\begin{equation}\label{cost}
\begin{split}
\zeta(k,z):=\overline\zeta(kz,z)\,,\quad
\eta(k,z):= \overline \eta(kz,z)\,, \quad
\psi(k,z):=\overline \psi(kz,z)\,, \quad
\xi(k,z):= \overline \xi(kz,z),
\end{split}
\end{equation}
where $\overline\zeta, \overline\eta, \overline\psi, \overline\xi$ are given in \eqref{cost1}. See the proof of Lemma \ref{parita} in the Appendix for the explicit form of the above functions.
Putting into \eq{soluzdirac2} $z=m\ell>0$, $y=k\ell$ with $k\in[-1,1]$ and $w=s\ell$ with $s\in[-1,1]$, each $\phi_m(y,w)$ rewrites as the three variable function:
\begin{equation}\label{soluzdiracz}
\begin{split}
\phi(s,k,z)= e^{-z}g(s,k,z)+h(s,k,z),
\end{split}
\end{equation} 
where
\begin{equation*}
g(s,k,z):=\cosh(sz)\bigg(\dfrac{\zeta(k,z)}{F(z)}+z\dfrac{\psi(k,z) }{F(z)}-sz\dfrac{\xi(k,z)}{\overline F(z)}\bigg)+\sinh(sz)\bigg(\dfrac{\eta(k,z)}{\overline F(z)}+z\dfrac{\xi(k,z) }{\overline F(z)}-sz\dfrac{\psi(k,z)}{F(z)}\bigg)\\
\end{equation*}
and
\begin{equation*}
h(s,k,z):=(1+z|k-s|)e^{-z|k-s|}\,.
\end{equation*}
 It is readily seen that the monotonicity issue \eqref{decreasing} follows by showing that the function $\phi(s,k,z)$ is decreasing with respect to $z>0$ for all $k,s\in[-1,1]$, i.e. 
$$
 \phi_z(s,k,z)=e^{-z}\big(g_z(s,k,z)-g(s,k,z)\big)+h_{z}(s,k,z)<0\qquad \forall z>0, \forall k,s \in[-1,1].
$$
 Since $h_{z}(s,k,z)=-(k-s)^2ze^{-z|k-s|} \leq 0$, for all $z>0$ and $k,s\in[-1,1]$, a sufficient condition for the validity of the above inequality is:
 \begin{equation}\label{g_monotone}
 g_{z}(s,k,z)-g(s,k,z)<0\qquad \forall z>0, \forall k,s \in[-1,1].
 \end{equation}
 The proof of this inequality will be the goal of this section. To this aim we compute 
$$
g_{z}(s,k,z)-g(s,k,z)= W(s,k,z)\cosh(sz)+Q(s,k,z)\sinh(sz)
$$
in which we set
\small
$$
W(s,k,z):=\bigg[\dfrac{\zeta(k,z)}{F(z)}+z\dfrac{\psi(k,z)}{F(z)}\bigg]_z-\bigg[\dfrac{\zeta(k,z)}{F(z)}+z\dfrac{\psi(k,z)}{F(z)}\bigg]+s\bigg(\dfrac{\eta(k,z)}{\overline F(z)}+2z\dfrac{\xi(k,z)}{\overline F(z)}-\bigg[z\dfrac{\xi(k,z)}{\overline F(z)}\bigg]_z\bigg)-s^2z\dfrac{\psi(k,z)}{F(z)}
$$
and
$$
Q(s,k,z):=\bigg[\dfrac{\eta(k,z)}{\overline F(z)}+z\dfrac{\xi(k,z)}{\overline F(z)}\bigg]_z-\bigg[\dfrac{\eta(k,z)}{\overline F(z)}+z\dfrac{\xi(k,z)}{\overline F(z)}\bigg]+s\bigg(\dfrac{\zeta(k,z)}{ F(z)}+2z\dfrac{\psi(k,z)}{ F(z)}-\bigg[z\dfrac{\psi(k,z)}{F(z)}\bigg]_z\bigg)-s^2z\dfrac{\xi(k,z)}{\overline F(z)}.
$$
\normalsize

In view of the elementary implication:
\begin{equation}\label{lemma0}
W(z)+|Q(z)|<0 \quad\Rightarrow\quad  W(z)\cosh(\omega z)+Q(z)\sinh(\omega z)<0\qquad\forall z>0,\forall \omega \in\mathbb{R}
\end{equation}
for all $W, Q:(0,+\infty)\rightarrow\mathbb{R}$ continuous functions, it follows that a sufficient condition for \eqref{g_monotone} to hold is 
\begin{equation}\label{suff}
W(s,k,z)+Q(s,k,z)<0 \quad   \wedge \quad W(s,k,z)-Q(s,k,z)<0 \quad \forall z>0\,,\forall k\,s\in[-1,1]\,.
\end{equation}
 We consider
\begin{equation}\label{parabolas}
\begin{split}
&W(s,k,z)+Q(s,k,z)=-s^2z\bigg[\dfrac{\psi(k,z)}{F(z)}+\dfrac{\xi(k,z)}{\overline F(z)}\bigg]\\
&+s\bigg(\dfrac{\eta(k,z)}{\overline F(z)}+\dfrac{\zeta(k,z)}{ F(z)}+2z\dfrac{\xi(k,z)}{\overline F(z)}+2z\dfrac{\psi(k,z)}{ F(z)}-\bigg[z\dfrac{\xi(k,z)}{\overline F(z)}+z\dfrac{\psi(k,z)}{F(z)}\bigg]_z\bigg)\\
&+\bigg[\dfrac{\zeta(k,z)}{F(z)}+z\dfrac{\psi(k,z)}{F(z)}+\dfrac{\eta(k,z)}{\overline F(z)}+z\dfrac{\xi(k,z)}{\overline F(z)}\bigg]_z-\bigg[\dfrac{\zeta(k,z)}{F(z)}+z\dfrac{\psi(k,z)}{F(z)}+\dfrac{\eta(k,z)}{\overline F(z)}+z\dfrac{\xi(k,z)}{\overline F(z)}\bigg]\\
&W(s,k,z)-Q(s,k,z)=-s^2z\bigg[\dfrac{\psi(k,z)}{F(z)}-\dfrac{\xi(k,z)}{\overline F(z)}\bigg]\\
&+s\bigg(\dfrac{\eta(k,z)}{\overline F(z)}-\dfrac{\zeta(k,z)}{ F(z)}+2z\dfrac{\xi(k,z)}{\overline F(z)}-2z\dfrac{\psi(k,z)}{ F(z)}-\bigg[z\dfrac{\xi(k,z)}{\overline F(z)}-z\dfrac{\psi(k,z)}{F(z)}\bigg]_z\bigg)\\
&+\bigg[\dfrac{\zeta(k,z)}{F(z)}+z\dfrac{\psi(k,z)}{F(z)}-\dfrac{\eta(k,z)}{\overline F(z)}-z\dfrac{\xi(k,z)}{\overline F(z)}\bigg]_z-\bigg[\dfrac{\zeta(k,z)}{F(z)}+z\dfrac{\psi(k,z)}{F(z)}-\dfrac{\eta(k,z)}{\overline F(z)}-z\dfrac{\xi(k,z)}{\overline F(z)}\bigg]\,.
\end{split}
\end{equation}
The maps $[-1,1]\ni s\mapsto W(s,k,z)\pm Q(s,k,z)$ are concave parabolas for all $z>0$ and $k\in[-1,1]$ fixed. Indeed, we have
\begin{equation}\label{dis1}
\dfrac{\psi(k,z)}{F(z)}\pm\dfrac{\xi(k,z)}{\overline F(z)}>0\qquad\forall z>0, k\in[-1,1].
\end{equation}
Furthermore, there holds:

\begin{align}\label{vertice}
	\bigg[z\dfrac{\psi(k,z)}{F(z)}\pm z\dfrac{\xi(k,z)}{\overline F(z)}\bigg]_z-\bigg[\dfrac{\zeta(k,z)}{F(z)}\pm\dfrac{\eta(k,z)}{\overline F(z)}\bigg]< 0\qquad &k\in[-1,1],\forall z>0
	\end{align}
	and
	\begin{align}\label{maximum}
	 \bigg[\dfrac{\zeta(k,z)}{F(z)}\pm\dfrac{\eta(k,z)}{\overline F(z)}\bigg]_z<0\qquad &\forall k\in[-1,1],\forall z>0.
\end{align}

The first condition assures that the abscissa $\overline s$ of the parabolas vertex satisfies, respectively, $\overline s>1$ or $\overline s<-1$, implying that the maximum is achieved, respectively, at $s=1$ or at $s=-1$; condition \eq{maximum} implies the negativity of such maxima proving \eqref{suff} and, in turn, \eqref{g_monotone}. We postpone the (long) proofs of \eq{dis1}, \eq{vertice} and \eq{maximum}, respectively, to Sections \ref{varie}, \ref{varie1} and \ref{varie2} below. \par 
\begin{remark}\label{proof idea}
It's worth pointing out that the proofs of \eq{vertice} and \eq{maximum} are achieved by repeating several times the scheme outlined above, i.e. we first put in evidence an expression of the type: $W\cosh(\omega z)+Q\sinh(\omega z)$, for suitable functions $W$ and $Q$, and then, in order to show that this expression is always negative, we exploit \eq{lemma0} and we come to study the sign of $W\pm Q$. As in \eq{parabolas}, the functions $W\pm Q$ can always be seen as parabolas with respect to one of the variables: we locate the maximum point of these parabolas and we estimate the sign of the maximum in a suitable interval. The advantage of this procedure is that, at each step, we obtain a reduction of the number of variables. Indeed, we start with the three variables functions $W$ and $Q$ in \eq{parabolas} and we reduce to two or one variables functions, see e.g. \eq{ts} below.
\end{remark}

\begin{remark}\label{knegative}
	Except for \eq{vertice}, all steps in the proof of the monotonicity issue  \eq{decreasing} hold for all $\sigma\in(-1,1)$. Our numerical experiments suggest that \eq{suff} is still satisfied when $\sigma \in(-1,0)$ but the vertex of the parabolas $s\mapsto W(s,k,z)\pm Q(s,k,z)$ in \eq{parabolas}, differently to what happens for $\sigma \in[0,1)$, may belong to the interval $[-1,1]$. Therefore, to extend the proof to the case $\sigma \in(-1,0)$, condition \eq{maximum} should be modified accordingly.
\end{remark}

	\subsection{Proof of the positivity issue in \eqref{decreasing}.}
	By \eq{soluzdiracz} the sign of $\phi_m(y,w)$ is the same of
the function $\phi(s,k,z)=e^{-z}g(s,k,z)+h(s,k,z)$. Since $h(s,k,z)>0$ for all $z>0$, $k,s\in[-1,1]$, to obtain the positivity of $\phi_m$ in \eqref{decreasing} we prove that $e^{-z}g(s,k,z)\geq 0$ for all $z>0$, $k,s\in[-1,1]$.\par
	For $z\rightarrow +\infty$, by direct inspection, we get that $e^{-z}g(s,k,z)\rightarrow L$ with $L=0$ for all $k,s\neq \pm 1$ and $L=\frac{4+(1+\sigma)^2}{(1-\sigma)^2}$ if $k=s=\pm 1$. Therefore, the strict monotonicity of $e^{-z}g(s,k,z)$ proved in Section \ref{monot} assures the positivity of $\phi(z,k,s)$, i.e. the positivity issue in \eqref{decreasing}.

\subsection{Proofs of inequality \eq{dis1}.}\label{varie}
Here and after, we will exploit the inequalities
\begin{equation}\label{dis2}
2\cosh(\omega z)-(1+\sigma)\sinh(\omega z)>0\qquad  \forall z>0\,, \forall \omega \in \R\,,
\end{equation}
\begin{equation}\label{ineqvarie}
\dfrac{1}{F(z)}\pm\dfrac{1}{\overline F(z)}>0\quad \text{and}\qquad \dfrac{1}{[F(z)]^2}\pm\dfrac{1}{[\overline F(z)]^2}>0\qquad  \forall z>0\,,
\end{equation}
where $F(z)$ and $\overline F(z)$ are as in \eq{cost1}. The proof of \eq{dis2} is immediate while inequality \eqref{ineqvarie} simply follows by noticing that $\overline F(z)>F(z)>0$ for all $z>0$.\par
Next we prove \eq{dis1}.
\begin{lemma}\label{lemma2}
	Given $F(z)$, $\overline F(z)$ as in \eq{cost1} and $\psi(k,z)$, $\xi(k,z)$ as in \eq{cost}, we have
\begin{equation}\label{dis1BIS}
\dfrac{\psi(k,z)}{F(z)}\pm\dfrac{\xi(k,z)}{\overline F(z)}>0\qquad\forall z>0, k\in[-1,1].
\end{equation}
\end{lemma}
\begin{proof}
We observe that
\begin{equation*}
\begin{split}
	\psi(k,z)=&\big(2+(1-\sigma)z\big)\cosh(kz)\cosh(z)+\big(-(1+\sigma)+z(1-\sigma)\big)\cosh(kz)\sinh(z)\\&-kz(1-\sigma)\sinh(kz)\cosh(z)-kz(1-\sigma)\sinh(kz)\sinh(z)>0  \qquad \forall z>0, k\in[-1,1],
\end{split}
\end{equation*} thanks to \eq{dis2} and $\cosh(kz)-k\sinh(kz)>0$ for all $z>0$ and $k\in[-1,1]$. Hence, through the first of \eq{ineqvarie}, a sufficient condition for the validity of \eq{dis1BIS} is $\psi(k,z)\pm\xi(k,z)>0$.
But by \eq{dis2} we immediately deduce
\begin{equation*}
\begin{split}
\psi+\xi=&\big(2+(1-\sigma)(1-k)z\big)\cosh[(1+k)z]+\big(-(1+\sigma)+z(1-\sigma)(1-k)\big)\sinh[(1+k)z]>0,\\
\psi-\xi=&\big(2+(1-\sigma)(1+k)z\big)\cosh[(1-k)z]+\big(-(1+\sigma)+z(1-\sigma)(1+k)\big)\sinh[(1-k)z]>0,
\end{split}
\end{equation*}
for all $z>0$ and $k\in[-1,1]$. This concludes the proof.
\end{proof}

\subsection{Proof of inequality \eq{vertice}.}\label{varie1}
The proof of \eq{vertice} is given in Lemma \ref{lemma6}.
\begin{lemma}\label{lemma6}
	Given $F(z)$, $\overline F(z)$ as in \eq{cost1} and $\zeta(k,z)$, $\eta(k,z)$, $\psi(k,z)$, $\xi(k,z)$ as in \eq{cost}, we have that
	\begin{align} \label{dis8}
	&\bigg[z\dfrac{\psi}{F}\pm z\dfrac{\xi}{\overline F}\bigg]_z-\bigg[\dfrac{\zeta}{F}\pm\dfrac{\eta}{\overline F}\bigg] = \dfrac{(\psi+z\psi_z-\zeta) F-z\psi F'}{F^2}\pm\dfrac{(\xi+z\xi_z-\eta)\overline F-z\xi \overline F'}{\overline F^2}<0
	\end{align}
	for all $k\in[-1,1]$ and for all $z>0$.
\end{lemma}
\begin{proof}
Since
\begin{align*}
\zeta(-k,z)&=\zeta(k,z),\quad &\psi(-k,z)&=\psi(k,z),\quad &\psi_z(-k,z)&=\psi_z(k,z),\quad &\forall k\in[-1,1], \forall z>0,\\
\eta(-k,z)&=-\eta(k,z),\quad &\xi(-k,z)&=-\xi(k,z),\quad &\xi_z(-k,z)&=-\xi_z(k,z)\quad &\forall k\in[-1,1], \forall z>0,
\end{align*}
see Lemma \ref{parita} in the Appendix for the explicit form of the above functions, the second term of \eq{dis8} is given by the sum of an even and an odd function with respect to $k$. Hence, to obtain \eq{dis8} it is enough to prove that \small \begin{equation}\label{ts+}
 \dfrac{(\psi+z\psi_z-\zeta) F(z)-z\psi F'}{F^2}+\dfrac{(\xi+z\xi_z-\eta)\overline F-z\xi \overline F'}{\overline F^2}<0\, \qquad \forall k\in[-1,1] \text{ and } z>0\,.
 \end{equation}\normalsize
    
We rewrite \eq{ts+} as
\begin{equation}\label{ts}
\cosh(kz)\,\mathcal{W}(k,z)+\sinh(kz)\,\mathcal{Q}(k,z)<0 \, \qquad \forall k\in[-1,1] \text{ and } z>0
\end{equation}
where
$$
\mathcal{W}(k,z):=k^2z^2\,s(z)+kz\,t(z)+u(z),\qquad \mathcal{Q}(k,z):=k^2z^2\,p(z)+kz\,q(z)+r(z)
$$
and
\begin{equation}\label{vari coeff}
\begin{split}
p(z):&=-\dfrac{(1-\sigma)}{\overline F(z)}[\cosh(z)+\sinh(z)]<0\\
q(z):&=\dfrac{1}{F(z)}\bigg(2(1+\sigma)\cosh(z)-4\sinh(z)+z(1-\sigma)[\cosh(z)+\sinh(z)]\dfrac{F'(z)}{F(z)}\bigg)\\
r(z):&=\dfrac{1}{\overline F(z)}\bigg\{[\cosh(z)+\sinh(z)]\bigg(-\dfrac{2(1+\sigma)}{1-\sigma}+2z(1-\sigma)+z^2(1-\sigma)\bigg)\\&\hspace{5mm}-z \bigg(\cosh(z)[-1-\sigma+z(1-\sigma)]+\sinh(z)[2+z(1-\sigma)]\bigg)\dfrac{\overline F'(z)}{\overline  F(z)}\bigg\}\\
s(z):&=-\dfrac{(1-\sigma)}{F(z)}[\cosh(z)+\sinh(z)]<0\\
t(z):&=\dfrac{1}{\overline F(z)}\bigg(-4\cosh(z)+2(1+\sigma)\sinh(z)+z(1-\sigma)[\cosh(z)+\sinh(z)]\dfrac{\overline F'(z)}{\overline F(z)}\bigg)\\
u(z):&=\dfrac{1}{ F(z)}\bigg\{[\cosh(z)+\sinh(z)]\bigg(-\dfrac{2(1+\sigma)}{1-\sigma}+2z(1-\sigma)+z^2(1-\sigma)\bigg)\\&\hspace{5mm}-z \bigg(\cosh(z)[2+z(1-\sigma)]+\sinh(z)[-1-\sigma+z(1-\sigma)]\bigg)\dfrac{F'(z)}{F(z)}\bigg\}.\\
\end{split}
\end{equation}
By \eqref{lemma0}, \eqref{ts} follows if $\chi^{\pm }(k,z):=\mathcal{W}(k,z)\pm \mathcal{Q}(k,z)<0$ for all $k\in[-1,1]$ and $z>0$, namely if
\begin{equation}\label{chi+}
\chi^+(k,z):=k^2z^2\,[s(z)+p(z)]+kz\,[t(z)+q(z)]+u(z)+r(z)<0 \qquad \forall k\in[-1,1] \text{ and }z>0
\end{equation}
and
\begin{equation}\label{chi-}
\chi^{-}(k,z):=k^2z^2\,[s(z)-p(z)]+kz\,[t(z)-q(z)]+u(z)-r(z)<0 \qquad \forall k\in[-1,1] \text{ and }z>0\,.
\end{equation}

 We prove the validity of \eq{chi+} and \eq{chi-} here below; this concludes the proof of Lemma \ref{lemma6}.

\par \bigskip\par

\textbf{Proof of \eq{chi+}.}

By  \eq{vari coeff}, $s(z)+p(z)<0$ for all $z>0$, hence $\chi^+(k,z)$ is a concave parabola with respect to $k$. Therefore, $\chi^+(k,z)<0$ if the ordinate of its vertex is negative, namely if
\begin{equation*}\label{maxP1}
	\dfrac{4[s(z)+p(z)][u(z)+r(z)]-[t(z)+q(z)]^2}{4[s(z)+p(z)]}=:\dfrac{\mu(z)}{4[s(z)+p(z)]}<0\quad \forall z>0.
\end{equation*}
Through many computations we obtain
\small 
\begin{equation*}
\begin{split}
\mu(z)=&2(1-\sigma)(3+\sigma)\bigg[\dfrac{1}{[F(z)]^2}-\dfrac{1}{[\overline F(z)]^2}+2\dfrac{z}{F(z)\overline F(z)}\bigg(\dfrac{F'(z)}{F(z)}-\dfrac{\overline F'(z)}{\overline F(z)}\bigg)\bigg]\\&+ \nero\dfrac{(3+\sigma)^2}{[F(z)\overline F(z)]^2}\bigg[\cosh(2z)\bigg((7+10\sigma-\sigma^2)\sinh^2(2z)-4(1-\sigma)^2z^2\bigg)\\&\hspace{20mm}+\sinh(2z)\bigg((7+10\sigma-\sigma^2)\sinh^2(2z)+4(1-\sigma)^2z^2\bigg)\bigg]\\&-[\cosh(2z)+\sinh(2z)]z(1-\sigma)^2\bigg(\dfrac{2F(z)-F'(z)}{[F(z)]^2}+\dfrac{2\overline F(z)-\overline F'(z)}{[\overline F(z)]^2}\bigg)\cdot\\&\hspace{20mm}\cdot\bigg[\dfrac{(4+2z)F(z)-zF'(z)}{[F(z)]^2}+\dfrac{(4+2z)\overline F(z)-z\overline F'(z)}{[\overline F(z)]^2}\bigg].
\end{split}
\end{equation*}
\normalsize
We have
\begin{equation*}
\dfrac{F'(z)}{F(z)}-\dfrac{\overline F'(z)}{\overline F(z)}=\dfrac{(3+\sigma)(1-\sigma)}{F(z)\overline F(z)}[2z\cosh(2z)-\sinh(2z)]>0 \quad \forall z>0
\end{equation*}
since $[2z\cosh(2z)-\sinh(2z)]'(z)=4z\sinh(2z)>0$ for all $z>0$. Hence, recalling \eq{ineqvarie}, the first term in the definition of $\mu$ is positive. Moreover, by estimating $\sinh(2z)>2z$ for $z>0$, we have
\begin{align*}
&(7+10\sigma-\sigma^2)\sinh^2(2z)-4(1-\sigma)^2z^2>8z^2(\sigma+2\sqrt{3}-3)(3+2\sqrt{3}-\sigma)>0\quad\forall z>0\\&(7+10\sigma-\sigma^2)\sinh^2(2z)+4(1-\sigma)^2z^2=(\sigma+4\sqrt{2}-5)(5+4\sqrt{2}-\sigma)\sinh^2(2z)+4(1-\sigma)^2z^2>0\quad\forall z>0.
\end{align*}
\normalsize
By Lemma \ref{AF_F1} in the Appendix we know that $\frac{2F(z)-F'(z)}{[F(z)]^2}+\frac{2\overline F(z)-\overline F'(z)}{[\overline F(z)]^2}<0$, therefore  if 
\begin{equation}\label{ts4}
	\dfrac{(4+2z)F(z)-zF'(z)}{[F(z)]^2}+\dfrac{(4+2z)\overline F(z)-z\overline F'(z)}{[\overline F(z)]^2}>0\qquad \forall z>0,
\end{equation}
then $\mu(z)>0$.
To this aim we consider $$\mu_1(z):=(4+2z)F(z)-zF'(z)=(3+\sigma)\{z[\sinh(2z)-\cosh(2z)]+2\sinh(2z)\}-(1-\sigma)(3z+2z^2);$$
since $\mu_1(0)=0$ and
\begin{align*}
\mu_1'(z)&=(3+\sigma)\{\cosh(2z)[3+2z]+\sinh(2z)[1-2z]\}-(1-\sigma)(3+4z)\\&>2\{(3+\sigma)z[\cosh(2z)-\sinh(2z)]+3+3\sigma+z(1+3\sigma)\}>0\qquad \forall z>0,
\end{align*}
 we have $\mu_1(z)>0$ for all $z>0$. On the other hand, we have
 \small
 $$
 \mu_2(z):=(4+2z)\overline F(z)-z\overline F'(z)=(3+\sigma)\{z[\sinh(2z)-\cosh(2z)]+2\sinh(2z)\}+(1-\sigma)(3z+2z^2)>\mu_1(z)>0\quad \forall z>0,
 $$
 \normalsize
 implying \eq{ts4}. This assures $\chi^+(k,z)<0$ for all $k\in[-1,1]$ and for all $z>0$.\par 

\par \bigskip\par

\textbf{Proof of \eq{chi-}.}

 First of all we notice that $\chi^-(k,z)$ is a concave parabola with respect to $k$, since $$s(z)-p(z)=-(1-\sigma)\bigg[\frac{1}{F(z)}-\dfrac{1}{\overline F(z)}\bigg][\cosh(z)+\sinh(z)]<0\qquad\forall z>0\,.$$
We prove that the parabola has a point of maximum for $k<-1$, i.e. that
\begin{equation}\label{ts3}
\overline \mu(z):=t(z)-q(z)+2z[p(z)-s(z)]< 0\qquad \forall z>0.
\end{equation}
To this aim we study
\begin{equation*}
\begin{split}
\overline \mu(z)=&\dfrac{2}{F(z)\overline F(z)}\bigg[\sinh(z)\big[(1+\sigma) F(z)+2\overline F(z)\big]-\cosh(z)\big[2F(z)+(1+\sigma)\overline F(z)\big]\bigg]\\&+z(1-\sigma)[\cosh(z)+\sinh(z)]\bigg[\dfrac{2F(z)-F'(z)}{[F(z)]^2}-\dfrac{2\overline F(z)-\overline F'(z)}{[\overline F(z)]^2}\bigg].
\end{split}
\end{equation*}
 By Lemma \ref{AF_F1} in the Appendix we have that the last term above is negative; about the remaining terms we distinguish the cases $z\in(0,1]$ and $z>1$.\par
  For $z\in(0,1]$ we have 
 \begin{equation*}
 \begin{split}
 &\sinh(z)\big[(1+\sigma) F(z)+2\overline F(z)\big]-\cosh(z)\big[2F(z)+(1+\sigma)\overline F(z)\big]\\&=\dfrac{(3+\sigma)^2}{2}\sinh(2z)[\sinh(z)-\cosh(z)]+z(1-\sigma)^2[\cosh(z)+\sinh(z)]\\&<2z\big[(\sigma^2+2\sigma+5)\sinh(z)-4(1+\sigma)\cosh(z)\big]:=v(\sigma).
 \end{split}
 \end{equation*}
 We observe that $\dfrac{d v}{d\sigma}=4z\big[(1+\sigma)\sinh(z)-2\cosh(z)\big]<0$, hence $v(\sigma)<2z[5\sinh(z)-4\cosh(z)]<0$ for $z<\log3$, implying $\overline \mu(z)<0$ for $z\in(0,1]$.\par For $z>1$ we rewrite 
 \begin{equation*}
\overline \mu(z)=\cosh(z) \overline\, W(z)+\sinh(z)\,\overline Q(z)\,,
 \end{equation*}
 where
 $$\overline W(z):=-\dfrac{4}{\overline F(z)}-\dfrac{2(1+\sigma)}{F(z)}+z(1-\sigma)\bigg(\dfrac{2F(z)-F'(z)}{[F(z)]^2}-\dfrac{2\overline F(z)-\overline F'(z)}{[\overline F(z)]^2}\bigg)$$
$$\overline Q(z):= \dfrac{2(1+\sigma)}{\overline F(z)}+\dfrac{4}{F(z)}+z(1-\sigma)\bigg(\dfrac{2F(z)-F'(z)}{[F(z)]^2}-\dfrac{2\overline F(z)-\overline F'(z)}{[\overline F(z)]^2}\bigg)$$
 
 and, by \eqref{lemma0}, we prove that $\overline \mu$ is negative by showing that $\overline W(z)\pm \overline Q(z)<0$ for $z>1$. 
 
 The case $\overline W(z)-\overline Q(z)<0$ is trivially true for all $z>0$, then it remains to study $$\overline W(z)+\overline Q(z)=z\dfrac{(1-\sigma)^2}{[F(z)\overline F(z)]^2}\bigg[(3+\sigma)^2\bigg((1+z)\cosh(4z)-2z\sinh(4z)-1-z\bigg)-8z^3(1-\sigma)^2\bigg].$$
 We consider $\overline \mu_1(z):=(1+z)\cosh(4z)-2z\sinh(4z)-1$ and $\overline \mu_1'(z)=(1-8z)\cosh(4z)+2(1+2z)\sinh(4z)$, so that $\overline \mu_1(z)$ has stationary points satisfying $\tanh(4\overline z)=(8\overline z-1)/[2(1+2\overline z)]:=\gamma(\overline z)$. We observe that $\gamma(\overline z)$ is always increasing for $\overline z>0$, $\gamma(\overline z)=1$ if and only if $\overline z=3/4<1$, implying $\overline \mu_1'(z)<0$ for $z>1$; since $\overline \mu_1 (1)=2e^{-4}-1<0$ then $\overline \mu_1(z)<0$ for $z>1$ and in conclusion $\overline W(z)+\overline Q(z)<0$ for all $z>1$. This proves \eq{ts3}.\par In view of \eq{ts3}, to obtain $\chi^-(k,z)<0$ for all $z>0$ and $k\in [-1,1]$ it is enough to study the sign of
\begin{equation*}
\begin{split}
\chi^-(-1,z)&=z^2\,\big[s(z)-p(z)]-z\,[t(z)-q(z)]+u(z)-r(z)\\&=-\dfrac{2(1+\sigma)}{1-\sigma}\bigg(\dfrac{1}{F(z)}-\dfrac{1}{\overline F(z)}\bigg)[\cosh(z)+\sinh(z)]\\&+z\bigg\{\cosh(z)\bigg[2\dfrac{2F(z)-F'(z)}{[F(z)]^2}+(1+\sigma)\frac{2\overline F(z)-\overline F'(z)}{[\overline F(z)]^2}\bigg]\\&-\sinh(z)\bigg[(1+\sigma)\dfrac{2F(z)-F'(z)}{[ F(z)]^2}+2\dfrac{2\overline F(z)-\overline F'(z)}{[\overline F(z)]^2}\bigg)\bigg]
\\&=-\dfrac{2(1+\sigma)}{1-\sigma}\bigg(\dfrac{1}{F(z)}-\dfrac{1}{\overline F(z)}\bigg)[\cosh(z)+\sinh(z)]\\
&+z\bigg\{\cosh(z)\,\widetilde{W}(z)+\sinh(z)\, \widetilde{Q}(z)\bigg\}
\end{split}
\end{equation*}
where
$$\widetilde{W}(z):=2\dfrac{2F(z)-F'(z)}{[F(z)]^2}+(1+\sigma)\frac{2\overline F(z)-\overline F'(z)}{[\overline F(z)]^2}$$
and
$$\widetilde{Q}(z):=-(1+\sigma)\dfrac{2F(z)-F'(z)}{[ F(z)]^2}-2\dfrac{2\overline F(z)-\overline F'(z)}{[\overline F(z)]^2}\,.$$
Recalling that $\overline F(z)>F(z)>0$ for all $z>0$, if
\begin{equation*}\label{tss}
\cosh(z)\,\widetilde{W}(z)+\sinh(z)\, \widetilde{Q}(z)<0\quad  \text{for all } z>0,
\end{equation*}
we conclude that $\overline \chi(-1,z)<0$ and the thesis. By \eqref{lemma0} this follows by showing that $\widetilde{W}(z)\pm \widetilde{Q}(z)<0$ for all $z>0$, namely that
$$
(1-\sigma)\bigg[\dfrac{2F(z)-F'(z)}{[F(z)]^2}-\frac{2\overline F(z)-\overline F'(z)}{[\overline F(z)]^2}\bigg]<0\qquad (3+\sigma)\bigg[\dfrac{2F(z)-F'(z)}{[F(z)]^2}+\frac{2\overline F(z)-\overline F'(z)}{[\overline F(z)]^2}\bigg]<0.
$$
These inequalities hold true for all $z>0$ thanks to Lemma \ref{AF_F1} in the Appendix.
\end{proof}

\par \medskip\par

\subsection{Proof of inequality \eq{maximum}.}\label{varie2}
The proof of \eq{maximum} is given in Lemma \ref{lemma7}. 
\begin{lemma}\label{lemma7}
	Given $F(z)$, $\overline F(z)$ as in \eq{cost1} and $F(z)$, $\overline F(z)$, $\zeta(k,z)$, $\eta(k,z)$ as in \eq{cost}, we have
	\begin{align}
	\bigg[\dfrac{\zeta}{F}\pm\dfrac{\eta}{\overline F}\bigg]_z=\dfrac{\zeta_z F-\zeta F'}{[F]^2}\pm \dfrac{\eta_z \overline F-\eta \overline F'}{[\overline F]^2}<0\qquad &\forall k\in[-1,1],\forall z>0\,.\label{dis0}
	\end{align}
\end{lemma}
\begin{proof}
	Since
	\begin{align*}
	\zeta(-k,z)&=\zeta(k,z),\, &\zeta_z(-k,z)&=\zeta_z(k,z),\quad &\forall k\in[-1,1],\,\, \forall z>0,\\
	\eta(-k,z)&=-\eta(k,z),\, &\eta_z(-k,z)&=-\eta_z(k,z),\quad &\forall k\in[-1,1],\, \, \forall z>0,
	\end{align*}
	see Lemma \ref{parita} in the Appendix for the explicit form of the above functions. The second term of \eq{dis0} is given by the sum of an even and an odd function with respect to $k$; hence, to obtain \eq{dis0} it is enough to prove that
	\begin{equation}\label{ts000}
	\dfrac{\zeta_z(k,z) F(z)-\zeta(k,z) F'(z)}{[F(z)]^2}+ \dfrac{\eta_z(k,z) \overline F(z)-\eta(k,z) \overline F'(z)}{[\overline F(z)]^2}<0,
	\end{equation}\normalsize
	for all $k\in[-1,1]$ and $z>0$. We rewrite \eq{ts000} as
	\begin{equation*}\label{ts++}
	\cosh(kz)\mathcal{V}(k,z) +\sinh(kz) \mathcal{P}(k,z)
	\end{equation*}
	where
	$$
	\mathcal{V}(k,z):=k^2\,a(z)+k\,b(z)+c(z),\qquad \mathcal{P}(k,z):=k^2\,d(z)+k\,e(z)+f(z),
	$$
	and
	\begin{equation}\label{vari coeff2}
	\begin{split}
	a(z):&=-\dfrac{z}{F(z)}[2\cosh(z)-(1+\sigma)\sinh(z)]<0\\
	b(z):&=\dfrac{1}{\overline F(z)}\bigg(\dfrac{2(1+\sigma)}{1-\sigma}[\cosh(z)+\sinh(z)]+z[2\sinh(z)-(1+\sigma)\cosh(z)]\dfrac{\overline F'(z)}{\overline F(z)}\bigg)\\
	c(z):&=\dfrac{1}{ F(z)}\bigg\{\cosh(z)\bigg(\dfrac{2\sigma(1+\sigma)}{1-\sigma}+2z\bigg)+\sinh(z)\bigg(\dfrac{2(3-\sigma)}{1-\sigma}-z(1+\sigma)\bigg)\\&\hspace{5mm}- \bigg[\cosh(z)\bigg(\dfrac{4}{1-\sigma}-z(1+\sigma)\bigg)+\sinh(z)\bigg(\dfrac{(1+\sigma)^2}{1-\sigma}+2z\bigg)\bigg]\dfrac{F'(z)}{  F(z)}\bigg\}\\
	d(z):&=-\dfrac{z}{\overline F(z)}[2\sinh(z)-(1+\sigma)\cosh(z)]\\
	e(z):&=\dfrac{1}{F(z)}\bigg(\dfrac{2(1+\sigma)}{1-\sigma}[\cosh(z)+\sinh(z)]+z[2\cosh(z)-(1+\sigma)\sinh(z)]\dfrac{F'(z)}{F(z)}\bigg)\\
	f(z):&=\dfrac{1}{\overline F(z)}\bigg\{\cosh(z)\bigg(\dfrac{2(3-\sigma)}{1-\sigma}-z(1+\sigma)\bigg)+\sinh(z)\bigg(\dfrac{2\sigma(1+\sigma)}{1-\sigma}+2z\bigg)\\&\hspace{5mm}- \bigg[\cosh(z)\bigg(\dfrac{(1+\sigma)^2}{1-\sigma}+2z\bigg)+\sinh(z)\bigg(\dfrac{4}{1-\sigma}-z(1+\sigma)\bigg)\bigg]\dfrac{\overline F'(z)}{ \overline F(z)}\bigg\}.
	\end{split}
	\end{equation}
	Then, by \eqref{lemma0}, we obtain the thesis if $\Xi(k,z)^{\pm}:=\mathcal{V}(k,z)\pm \mathcal{P}(k,z)<0$ for all $k\in[-1,1]$ and $z>0$, i.e. if
	\begin{equation}\label{xi+}
	\Xi^{+}(k,z):=k^2\,[a(z)+d(z)]+k\,[b(z)+e(z)]+c(z)+f(z)<0 \qquad \forall k\in[-1,1]\,, \forall z>0
	\end{equation}
	and
	\begin{equation}\label{xi-}
	\Xi^{-}(k,z):=k^2\,[a(z)-d(z)]+k\,[b(z)-e(z)]+c(z)-f(z)<0 \qquad \forall k\in[-1,1]\,, \forall z>0\,.
	\end{equation}
	
	We prove the validity of \eq{xi+} and \eq{xi-} here below. This concludes the proof of Lemma \ref{lemma7}.

\par \bigskip\par
\textbf{Proof of \eq{xi+}.}\par 

	We consider
	$$
	a(z)+d(z)=-z\bigg[\cosh(z)\bigg(\dfrac{2}{F(z)}-\dfrac{(1+\sigma)}{\overline F(z)}\bigg)+\sinh(z)\bigg(\dfrac{2}{\overline F(z)}-\dfrac{(1+\sigma)}{F(z)}\bigg)\bigg].
	$$
Since, from \eq{ineqvarie}, we have $$\dfrac{2}{F(z)}-\dfrac{(1+\sigma)}{\overline F(z)}\pm\bigg(\dfrac{2}{\overline F(z)}-\dfrac{(1+\sigma)}{F(z)}\bigg)=[2\mp(1+\sigma)]\bigg(\dfrac{1}{ F(z)}\pm\dfrac{1}{\overline F(z)}\bigg)>0\qquad \forall z>0,$$  by \eqref{lemma0} we infer that $a(z)+d(z)<0$; hence the map $k\mapsto \Xi^+(k,z)$ is a concave parabola for all $z>0$.
Now we prove that the parabola has the abscissa vertex at $k=\overline k$ with $\overline k>1$; this follows by showing that
\begin{equation}\label{vertice1}
b(z)+e(z)+2[a(z)+d(z)]>0\qquad \forall z>0.
\end{equation}
We have
\begin{equation*}
\begin{split}
&b(z)+e(z)+2[a(z)+d(z)]=\dfrac{2(1+\sigma)}{(1-\sigma)}[\cosh(z)+\sinh(z)]\bigg(\dfrac{1}{F(z)}+\dfrac{1}{\overline F(z)}\bigg)\\&+z\bigg\{\cosh(z)\bigg[2\bigg(\dfrac{F'(z)}{F(z)^2}-\dfrac{2}{F(z)}\bigg)-(1+\sigma)\bigg(\dfrac{\overline F'(z)}{\overline F(z)^2}-\dfrac{2}{\overline F(z)}\bigg)\bigg]\\&\hspace{6mm}+\sinh(z)\bigg[2\bigg(\dfrac{\overline F'(z)}{\overline F(z)^2}-\dfrac{2}{\overline F(z)}\bigg)-(1+\sigma)\bigg(\dfrac{F'(z)}{F(z)^2}-\dfrac{2}{F(z)}\bigg)\bigg]\bigg\}.
\end{split}
\end{equation*}
Through \eqref{lemma0}, \eq{vertice1} holds if
$$
[2\mp (1+\sigma)]\bigg(\dfrac{F'(z)-2F(z)}{F(z)^2}\pm\dfrac{\overline F'(z)-2\overline F(z)}{\overline F(z)^2}\bigg)>0;
$$
this condition is guaranteed for all $z>0$ by Lemma \ref{AF_F1} in the Appendix. Hence, the maximum of $\Xi^+(k,z)$ is achieved at $k=1$; we prove that $\Xi^+(1,z)=a(z)+d(z)+b(z)+e(z)+c(z)+f(z)<0$ for all $z>0$. To this aim, we consider
\begin{equation*}
\begin{split}
\Xi^+(1,z)=\dfrac{1}{F(z)^2}\bigg\{&\cosh(z)\bigg[\dfrac{2(1+\sigma)^2}{1-\sigma}F(z)+F'(z)\bigg(z(3+\sigma)-\dfrac{4}{1-\sigma}\bigg)\bigg]\\+&\sinh(z)\bigg[\dfrac{8}{1-\sigma}F(z)-F'(z)\bigg(z(3+\sigma)+\dfrac{(1+\sigma)^2}{1-\sigma}\bigg)\bigg]\bigg
\}\\
+\dfrac{1}{\overline F(z)^2}\bigg\{&\cosh(z)\bigg[\dfrac{8}{1-\sigma}\overline F(z)-\overline F'(z)\bigg(z(3+\sigma)+\dfrac{(1+\sigma)^2}{1-\sigma}\bigg)\bigg]\\+&\sinh(z)\bigg[\dfrac{2(1+\sigma)^2}{1-\sigma}\overline F(z)+\overline F'(z)\bigg(z(3+\sigma)-\dfrac{4}{1-\sigma}\bigg)\bigg]\bigg\}:=\dfrac{\varsigma(z)}{F(z)^2}+\dfrac{\overline \varsigma(z)}{\overline F(z)^2}.
\end{split}
\end{equation*}
 Recalling that $\overline F(z)>F(z)>0$ for all $z>0$, we obtain that $\Xi^+(1,z)<0$ by showing that $\varsigma(z)<0$ and $\varsigma(z)+\overline \varsigma(z)<0$. Through many computations we get
$$
\varsigma(z)=\dfrac{e^z}{2}\bigg[-(1-\sigma)^2z-4(1+\sigma)-4\dfrac{(3+\sigma)(1+\sigma)}{1-\sigma}e^{-2z}+(3+\sigma)^2ze^{-4z}\bigg]\,.
$$
Consider $\widetilde{\varsigma}(z):=-(1-\sigma)^2z-4(1+\sigma)+(3+\sigma)^2ze^{-4z}$, we have that $\widetilde{\varsigma}(0)=-4(1+\sigma)<0$ and $\widetilde{\varsigma}(z)\rightarrow-\infty $ as $z\rightarrow +\infty$; furthermore $\widetilde{\varsigma}'(z)=-(1-\sigma)^2+(3+\sigma)^2e^{-4z}(1-4z)$ so that $\widetilde{\varsigma}'(\overline z)=0$ if and only if $\overline ze^{-4\overline z}=\frac{e^{-4\overline z}}{4}-\frac{(1-\sigma)^2}{4(3+\sigma)^2}$. 
We have $$\widetilde{\varsigma}(\overline z)=-(1-\sigma)^2 \bigg(\overline z+\dfrac{1}{4}\bigg)-4(1+\sigma)+(3+\sigma)^2\frac{e^{-4\overline z}}{4}\,.$$
Since $-(1-\sigma)^2 (\overline z+\frac{1}{4})-4(1+\sigma)+(3+\sigma)^2\frac{e^{-4\overline z}}{4}<0$ for all $z>0$, we infer that $\widetilde{\varsigma}(\overline z)<0$ and, in turn, that
\begin{equation}\label{bound}
\widetilde{\varsigma}(z)=-(1-\sigma)^2z-4(1+\sigma)+(3+\sigma)^2ze^{-4z}<0\qquad \forall z>0\,.
\end{equation}
This yields $\varsigma(z)<0$ for all $z>0$.
On the other hand, we have
$$
\varsigma(z)+\overline \varsigma(z)=-4\frac{(1+\sigma)(3+\sigma)}{1-\sigma}e^{-z}<0 \qquad\forall z>0\,.
$$
By this we conclude that $\Xi^+(1,z)<0$ for all $z>0$ and, in turn, that $\Xi^+(k,z)<0$ for all $k\in[-1,1]$ and for all $z>0$.\par 
\bigskip\par

\textbf{Proof of \eq{xi-}.}\par

 We have
	$$
	a(z)-d(z)=-z\bigg[\cosh(z)\bigg(\dfrac{2}{F(z)}+\dfrac{(1+\sigma)}{\overline F(z)}\bigg)-\sinh(z)\bigg(\dfrac{2}{\overline F(z)}+\dfrac{(1+\sigma)}{F(z)}\bigg)\bigg]\,.
	$$
	Since $$\dfrac{2}{F(z)}+\dfrac{(1+\sigma)}{\overline F(z)}\pm\bigg(\dfrac{2}{\overline F(z)}+\dfrac{(1+\sigma)}{F(z)}\bigg)=[2\pm(1+\sigma)]\bigg(\dfrac{1}{ F(z)}\pm\dfrac{1}{\overline F(z)}\bigg)>0\qquad \forall z>0,$$ by \eqref{lemma0}, we infer that $a(z)-d(z)<0$ and the map $k\mapsto \Xi^-(k,z)$ is a concave parabola for all $z>0$.
	Now we prove that the abscissa $\overline k$ of the parabola vertex satisfies $\overline k<-1$, namely that
	\begin{equation}\label{vertice2}
	b(z)-e(z)+2[d(z)-a(z)]<0\qquad \forall z>0.
	\end{equation}
	We have that
	\begin{equation*}
	\begin{split}
	&b(z)-e(z)+2[d(z)-a(z)]=-\dfrac{2(1+\sigma)}{(1-\sigma)}[\cosh(z)+\sinh(z)]\bigg(\dfrac{1}{F(z)}-\dfrac{1}{\overline F(z)}\bigg)\\&+z\bigg\{\cosh(z)\bigg[2\bigg(\dfrac{2}{F(z)}-\dfrac{F'(z)}{F(z)^2}\bigg)+(1+\sigma)\bigg(\dfrac{2}{\overline F(z)}-\dfrac{\overline F'(z)}{\overline F(z)^2}\bigg)\bigg]\\&\hspace{6mm}+\sinh(z)\bigg[-2\bigg(\dfrac{2}{\overline F(z)}-\dfrac{\overline F'(z)}{\overline F(z)^2}\bigg)-(1+\sigma)\bigg(\dfrac{2}{F(z)}-\dfrac{F'(z)}{F(z)^2}\bigg)\bigg]\bigg\}.
	\end{split}
	\end{equation*}
	Through \eqref{lemma0}, \eq{vertice2} follows if
	$$
	[2\pm (1+\sigma)]\bigg(\dfrac{2F(z)-F'(z)}{F(z)^2}\pm\dfrac{2\overline F(z)-\overline F'(z)}{\overline F(z)^2}\bigg)<0;
	$$
	this condition is guaranteed for all $z>0$ by Lemma \ref{AF_F1} in the Appendix.\par Hence, by \eq{vertice2}, $\Xi^{-}(k,z)$ achieves its maximum at $k=-1$; we prove that $\Xi^{-}(-1,z)=a(z)-d(z)-b(z)+e(z)+c(z)-f(z)<0$ for all $z>0$. To this aim, we consider
	\begin{equation*}
	\begin{split}
	\Xi^{-}(-1,z)=\dfrac{1}{F(z)^2}\bigg\{&\cosh(z)\bigg[\dfrac{2(1+\sigma)^2}{1-\sigma}F(z)+F'(z)\bigg(z(3+\sigma)-\dfrac{4}{1-\sigma}\bigg)\bigg]\\+&\sinh(z)\bigg[\dfrac{8}{1-\sigma}F(z)-F'(z)\bigg(z(3+\sigma)+\dfrac{(1+\sigma)^2}{1-\sigma}\bigg)\bigg]\bigg
	\}\\
	-\dfrac{1}{\overline F(z)^2}\bigg\{&\cosh(z)\bigg[\dfrac{8}{1-\sigma}\overline F(z)-\overline F'(z)\bigg(z(3+\sigma)+\dfrac{(1+\sigma)^2}{1-\sigma}\bigg)\bigg]\\+&\sinh(z)\bigg[\dfrac{2(1+\sigma)^2}{1-\sigma}\overline F(z)+\overline F'(z)\bigg(z(3+\sigma)-\dfrac{4}{1-\sigma}\bigg)\bigg]\bigg\}:=\dfrac{\varsigma(z)}{F(z)^2}-\dfrac{\overline \varsigma(z)}{\overline F(z)^2},
	\end{split}
	\end{equation*}
where $ \varsigma(z)$ and $\overline \varsigma(z)$ are as defined in the proof of \eq{xi+}. We have already proved that $\varsigma(z)<0$ for all $z>0$, by \eq{bound} we also get 
	$$
	\varsigma(z)-\overline \varsigma(z)=e^{z}\,\widetilde{\varsigma}(z)<0 \qquad\forall z>0\,.
	$$
	Hence, through \eq{ineqvarie} we deduce that $\Xi^{-}(-1,z)<0$ for all $z>0$. This assures $\Xi^{-}(k,z)<0$ for all $k\in[-1,1]$ and for all $z>0$ and concludes the proof of \eq{xi-}.

\end{proof}

 \section{Proof of Theorem \ref{greenpositiva}}\label{proof2}
 The proof is achieved by showing the positivity of the Green function $G_p(q)$ for $p$ and $q$ belonging to suitable rectangles or union of rectangles covering $\widetilde \Omega$. By Theorem \ref{monotonia}, we know that
\begin{equation*}
	G_p(q)=\sum_{m=1}^{+\infty} \dfrac{1}{2\pi}\dfrac{\phi_m(y,w)}{m^3}\sin(m\rho)\sin(mx) \qquad \forall q=(x,y)\in\overline \Omega,\,\,\,\forall p=(\rho,w)\in\overline \Omega.
	\end{equation*}
 In this section we will omit the dependence of $\phi_m$ from $y$ and $w$, implying that all relations we write hold true for all $y,w\in[-\ell,\ell]$; for this reason and for brevity, in all the proofs of this section we often write $G(x,\rho)$ instead of $G_p(q)=G(x,y,\rho,w)$. \par
 We start by showing the positivity of $G_p(q)$ for $p$ or $q$ far from the hinged edges of $\Omega$.
 
 \begin{proposition}\label{prop51}
$G_p(q)> 0$ if
\begin{equation*}
(q=(x,y)\in\bigg[\dfrac{\pi}{4},\dfrac{3}{4}\pi\bigg]\times[-\ell,\ell]\, \wedge\, p\in\widetilde\Omega)\quad \text{or} \quad (q\in\widetilde\Omega\, \wedge\, p=(\rho,w)\in\bigg[\dfrac{\pi}{4},\dfrac{3}{4}\bigg]\times[-\ell,\ell])\,.
\end{equation*}
 \end{proposition}
 \begin{proof}

 Thanks to Theorem \ref{monotonia}-\eqref{decreasing} we know that
 $$
 0<\phi_{m}<\phi_1\qquad\forall m>1.
 $$
Noting that $|\sin(m \rho)|<m\sin(\rho )$ for all $\rho \in(0,\pi)$, see e.g. \cite{Hardy-Fourier}, we obtain
$$
\dfrac{\phi_m}{m^3}|\sin(m\rho)\sin(mx)|\leq \dfrac{\phi_1}{m^2}\sin(\rho)|\sin(mx)|\qquad \forall x,\rho\in(0,\pi),
$$
from which 
$$
\sum_{m=2}^\infty\dfrac{\phi_m}{m^3}|\sin(m\rho)\sin(mx)|\leq \phi_1\sin(\rho)\sum_{m=2}^\infty\dfrac{|\sin(mx)|}{m^2}\qquad \forall x,\rho\in(0,\pi).
$$
Since
\begin{equation*}\label{dis}
	\sum_{m=2}^\infty\dfrac{|\sin(mx)|}{m^2}< \sum_{m=2}^\infty\dfrac{1}{m^2}=\frac{\pi^2}{6}-1\qquad \forall x\in(0,\pi),
\end{equation*}
we infer that
\begin{equation*}
\sum_{m=2}^\infty\dfrac{\phi_m}{m^3}\sin(m\rho)\sin(mx)>- \phi_1\sin(\rho)\bigg(\frac{\pi^2}{6}-1\bigg)\qquad \forall x,\rho\in(0,\pi)
\end{equation*}
and, in turn,
\begin{equation}\label{ppp centrale}
G(x,\rho)>\dfrac{\phi_1}{2\pi}\sin(\rho)\left[\sin(x)-\bigg(\frac{\pi^2}{6}-1\bigg)\right]\qquad \forall x,\rho\in(0,\pi).
\end{equation}
We denote by 
\begin{equation*}\label{x1}
x_1:=\arcsin\bigg(\frac{\pi^2}{6}-1\bigg)\approx 0.70<\dfrac{\pi}{4};
\end{equation*}
hence, through \eqref{ppp centrale} we have $G_p(q)>0$ for all $q=(x,y)\in(x_1,\pi-x_1)\times[-\ell,\ell]\, \wedge \, p\in\widetilde\Omega$, implying 
$$
G_p(q)>0\qquad\forall q=(x,y)\in\bigg[\dfrac{\pi}{4},\dfrac{3}{4}\pi\bigg]\times[-\ell,\ell]\, \wedge \, p\in\widetilde\Omega.
$$ 
The positivity in the region $q\in\widetilde\Omega\, \wedge\, p=(\rho,w)\in\big[\frac{\pi}{4},\frac{3}{4}\pi\big]\times[-\ell,\ell]$ follows by repeating the above proof with $x$ and $\rho$ reversed.
\end{proof}

Our next aim is to show the positivity issue for both $p$ and $q$ near the same hinged edge, i.e. near $x=0$ and $\rho=0$ or near $x=\pi$ and $\rho=\pi$. The proof is based on a multi step procedure; the first step is given by the following:
 \begin{lemma}\label{lemma1}
  Fix $N\geq 3$ integer, $G_p(q)> 0$ if
  \begin{equation*}
q=(x,y)\in \bigg[\dfrac{\pi}{N+2},\dfrac{\pi}{N+1}\bigg)\times[-\ell,\ell]\, \wedge \, p=(\rho,w)\in \bigg(0,\dfrac{\pi}{N+1}\bigg)\times[-\ell,\ell]\,.
\end{equation*}
 \end{lemma}
 \begin{proof}
 We fix $N\geq 3$ and we rewrite $G_p(q)$ as follows
$$
G(x,\rho)=\dfrac{1}{2\pi}\sum_{m=1}^N  \dfrac{\phi_m}{m^3}\sin(mx)\sin(m\rho)+\dfrac{1}{2\pi}\sum_{m=N+1}^\infty  \dfrac{\phi_m}{m^3}\sin(mx)\sin(m\rho)\,.
$$
Then, we exploit the elementary inequality 
$$
 	\sin(mx)\sin(m\rho)>\sin(x)\sin(\rho)\qquad \forall x,\rho\in\bigg(0,\dfrac{\pi}{N+1}\bigg),\quad \forall m=2,\dots,N
 	$$
(see Lemma \ref{lemmasin} in the Appendix for a proof) and Theorem \ref{monotonia}-\eqref{decreasing} to get
\begin{equation}\label{dise1}
\sum_{m=1}^N  \dfrac{\phi_m}{m^3}\sin(mx)\sin(m\rho)> \phi_N\sin(x)\sin(\rho) \sum_{m=1}^N  \dfrac{1}{m^3}\qquad \forall x,\rho\in\bigg(0,\dfrac{\pi}{N+1}\bigg).
\end{equation}
On the other hand, through $|\sin(m\rho)|<m\sin(\rho)$ for all $\rho\in(0,\pi)$ and Theorem \ref{monotonia}-\eqref{decreasing}, we get 
\begin{equation}\label{dise22}
\bigg|\sum_{m=N+1}^\infty  \dfrac{\phi_m}{m^3}\sin(mx)\sin(m\rho)\bigg|\leq \sum_{m=N+1}^\infty  \dfrac{\phi_m}{m^3}|\sin(mx)\sin(m\rho)|< \phi_{N}\sin(\rho)\sum_{m=N+1}^\infty  \dfrac{1}{m^2}.
\end{equation}
By combining \eqref{dise1} and \eqref{dise22} we infer 
\begin{equation}\label{ineq1}
\begin{split}
G(x,\rho)&>\dfrac{\phi_N}{2\pi}\sin(\rho) \left(\sum_{m=1}^N  \dfrac{1}{m^3}\right)\, \left[\sin(x)-C_N \right]\qquad \forall x,\rho\in \bigg(0,\dfrac{\pi}{N+1}\bigg)\,,
\end{split}
\end{equation}
where 
\begin{equation}\label{CNdef}
C_N:=\dfrac{\sum\limits_{m=N+1}^\infty  \dfrac{1}{m^2}}{\sum\limits_{m=1}^N  \dfrac{1}{m^3}}\,.
\end{equation}
Next we denote by $x_N$ the unique solution to the equation
\begin{equation*}\label{xN}
\sin(x_N)=C_N\, \qquad x_N\in (0,\pi/2) \,;
\end{equation*}
the above definition makes sense for all $N\geq 1$ since the map $N\mapsto C_N$ is positive, strictly decreasing and $0<C_N<1$. We prove that
\begin{equation}\label{ts0}
x_N< \dfrac{\pi}{N+2}\qquad \forall N\geq 3\,.
\end{equation}
When $N=3$  we have $x_3\approx0.25<\frac{\pi}{5}$ and \eqref{ts0} follows. We complete the proof of \eqref{ts0} by showing that
\begin{equation}\label{ts1}
C_N< \sin\bigg(\dfrac{\pi}{N+2}\bigg)\qquad \forall N\geq 4\,.
\end{equation}

To this purpose we write some estimates on the numerical series; it is easy to see that
$$
\sum\limits_{m=1}^N  \dfrac{1}{m^3}> 1\quad\text{and}\quad\sum\limits_{m=N+1}^\infty  \dfrac{1}{m^2}< \int_N^\infty\dfrac{1}{x^2}\,dx=\dfrac{1}{N} \qquad \forall N\geq 2,
$$
implying
\begin{equation}\label{CN}
C_N< \dfrac{1}{N}\qquad \forall N\geq 2.
\end{equation}
To tackle \eq{ts1} we use the estimate:
\begin{equation}\label{stimaseno}	
\sin(x)\geq\dfrac{3}{\pi}x\qquad \forall x \in\bigg(0,\dfrac{\pi}{6}\bigg].
\end{equation}
Combining \eq{CN} and \eq{stimaseno}, \eq{ts1} follows by noticing that
$
\dfrac{1}{N}<\dfrac{3}{N+2}$ for all $N\geq 4$. Finally, in view of \eq{ts0} the statement readily comes from the positivity of the r.h.s. of \eqref{ineq1}.

\end{proof}
Lemma \ref{lemma1} guarantees the positivity of $G(x,\rho)$ when $x\in  [\frac{\pi}{N+2},\frac{\pi}{N+1})$ and $\rho$ is closed to $0$. In the next lemma we prove that the statement still holds for all $\rho\in(0,\pi/4)$.

 \begin{lemma}\label{lemma11}
	Fixed $N\geq 3$, integer, $G_p(q)> 0$ if
	\begin{equation*}
	q=(x,y)\in \bigg[\dfrac{\pi}{N+2},\dfrac{\pi}{N+1}\bigg)\times[-\ell,\ell]\, \wedge \, p=(\rho,w)\in \bigg(0,\dfrac{\pi}{4}\bigg)\times[-\ell,\ell]\,.
	\end{equation*}
\end{lemma}
\begin{proof}
	The case $N=3$ is included in the statement of Lemma \ref{lemma1}.\par 
	When $N\geq4$, by Lemma \ref{lemma1} we know that
\begin{equation}\label{eq1}
	G(x,\rho)>0\qquad\forall x\in\bigg[\dfrac{\pi}{N+2},\dfrac{\pi}{N+1}\bigg),\,\,\forall\rho\in\bigg(0,\dfrac{\pi}{N+1}\bigg),
\end{equation}
with $\frac{\pi}{N+1}<\frac{\pi}{4}$. Moreover, with $C_{N}$ as in \eqref{CNdef}, we estimate 
\begin{equation}\label{ineq2}
\begin{split}
	G(x,\rho)&=\dfrac{1}{2\pi}\bigg[\sum_{m=1}^{N-1}  \dfrac{\phi_m}{m^3}\sin(mx)\sin(m\rho)+\sum_{m=N}^\infty  \dfrac{\phi_m}{m^3}\sin(mx)\sin(m\rho)\bigg]\\&>\dfrac{\phi_{N-1}}{2\pi}\sin(x) \left(\sum_{m=1}^{N-1}  \dfrac{1}{m^3}\right)\, \left[\sin(\rho)-C_{N-1} \right]\hspace{15mm}\qquad \,\,\forall x,\rho\in \bigg(0,\dfrac{\pi}{N}\bigg),
\end{split}
\end{equation}
in which we used Lemma \ref{lemmasin} of the Appendix with $N-1$ (instead of $N$), $|\sin(mx)|< m\sin(x)$ for all $x\in(0,\pi)$ and Theorem \ref{monotonia}-\eqref{decreasing}. In the following, for $N\geq 4$, we denote by $\rho_{N}$ the unique solution to the equation $\sin(\rho_{N})=C_{ N-1}$, with $C_{N}$ as in \eqref{CNdef}. Clearly, $\rho_N=x_{N-1}$ and by \eqref{ts0} we know that
\begin{equation}\label{ts00}
\rho_N< \dfrac{\pi}{N+1}\qquad \forall N\geq 4\,.
\end{equation}
Hence, through \eqref{ineq2} we have 
\begin{equation*}\label{eq2N}
G_p(q)>0\qquad\forall q=(x,y)\in\bigg(0,\dfrac{\pi}{N}\bigg)\times[-\ell,\ell]\, \wedge \, p=(\rho,w)\in\bigg(\rho_N,\dfrac{\pi}{N}\bigg)\times[-\ell,\ell] \qquad \forall N\geq 4,
\end{equation*}
implying by \eqref{ts00} that
\begin{equation}\label{eq2}
G_p(q)>0\qquad\forall q=(x,y)\in\bigg(0,\dfrac{\pi}{N}\bigg)\times[-\ell,\ell]\, \wedge \, p=(\rho,w)\in\bigg[\dfrac{\pi}{N+1},\dfrac{\pi}{N}\bigg)\times[-\ell,\ell]\qquad \forall N\geq 4.
\end{equation}
By combining \eqref{eq1}-\eqref{eq2}, since $[\frac{\pi}{N+2},\frac{\pi}{N+1}) \subset (0,\frac{\pi}{N})$,  we get
\begin{equation}\label{eq3}
G(x,\rho)>0\qquad\forall x\in\bigg[\dfrac{\pi}{N+2},\dfrac{\pi}{N+1}\bigg),\,\,\forall\rho\in\bigg(0,\dfrac{\pi}{N}\bigg)\,.
\end{equation}
If $N=4$ the above inequality yields the thesis since $\dfrac{\pi}{N}=\dfrac{\pi}{4}$.\par 
Let now $\overline N\geq 5$ fixed, by \eqref{eq3} we clearly have
\begin{equation*}
\begin{split}
G(x,\rho)>0\qquad\forall x\in\bigg[\dfrac{\pi} {\overline { N}+2},\dfrac{\pi}{ \overline { N}+1}\bigg),\,\,\forall\rho\in\bigg(0,\dfrac{\pi}{ \overline { N}}\bigg);
\end{split}
\end{equation*}
moreover, since $[\frac{\pi}{ \overline { N}+2},\frac{\pi}{ \overline { N}+1}) \subset (0,\frac{\pi}{N})$  for all $N\leq \overline N-1$, through \eqref{eq2} we also have
\begin{equation*}
\begin{split}
G(x,\rho)>0\qquad\forall x\in\bigg[\dfrac{\pi}{ \overline { N}+2},\dfrac{\pi}{ \overline { N}+1}\bigg),\,\,\forall\rho\in\bigg[\dfrac{\pi}{N+1},\dfrac{\pi}{N}\bigg)\qquad \forall N= 4,..., \overline { N}-1
\end{split}
\end{equation*}
and the thesis follows by noticing that $ \bigg(0,\dfrac{\pi}{ \overline { N}}\bigg) \bigcup \left( \displaystyle{\bigcup_{N=4}^{\overline N-1}} \bigg[\dfrac{\pi}{N+1},\dfrac{\pi}{N}\bigg)\right)= \bigg(0,\dfrac{\pi}{4}\bigg)$.
\end{proof}
Finally, from Lemma \ref{lemma11} we derive
 \begin{proposition}\label{prop2}
	$G_p(q)> 0$ if
	\begin{equation*}
	q=(x,y)\in \bigg(0,\dfrac{\pi}{4}\bigg)\times[-\ell,\ell]\, \wedge \, p=(\rho,w)\in \bigg(0,\dfrac{\pi}{4}\bigg)\times[-\ell,\ell]\,
	\end{equation*}
	\begin{center}
		or
	\end{center}
	\begin{equation*}
	q=(x,y)\in \bigg(\dfrac{3}{4}\pi,\pi\bigg)\times[-\ell,\ell]\, \wedge \, p=(\rho,w)\in \bigg(\dfrac{3}{4}\pi,\pi\bigg)\times[-\ell,\ell]\,.
	\end{equation*}
\end{proposition}
 \begin{proof}
 	By Lemma \ref{lemma11} we infer
 	\begin{equation*}\label{eq8}
 	G_p(q)>0\qquad\forall q=(x,y)\in\bigg[\dfrac{\pi}{N+2},\dfrac{\pi}{4}\bigg)\times[-\ell,\ell]\, \wedge \, p=(\rho,w)\in\bigg(0,\dfrac{\pi}{4}\bigg)\times[-\ell,\ell]\quad \forall N\geq 3.
 	\end{equation*}
 	Passing to the limit as  $N\rightarrow+\infty$ we obtain 
 	\begin{equation}\label{eq9}
 	G_p(q)>0\qquad\forall q=(x,y)\in\bigg(0,\dfrac{\pi}{4}\bigg)\times[-\ell,\ell]\, \wedge \, p=(\rho,w)\in\bigg(0,\dfrac{\pi}{4}\bigg)\times[-\ell,\ell].
 	\end{equation}
 	 	
 	To complete the proof we exploit some trigonometric relations. In particular, when $\tau \in \big(\frac{3}{4}\pi,\pi\big)$ we set $\overline \tau=\pi-\tau$ with $\overline \tau\in\big(0,\frac{\pi}{4}\big)$, obtaining
 	\begin{equation}\label{goniom}
 	\sin(m \tau)=
 	\begin{cases}
 	\sin(m\overline \tau)\quad &\text{for }m \text{ odd} \\
 	-\sin(m\overline \tau)\quad &\text{for }m \text{ even}. 
 	\end{cases}
 	\end{equation}
 	Then, for $x\in\big(\frac{3}{4}\pi,\pi\big)$ and $\rho\in\big(\frac{3}{4}\pi,\pi\big)$ we set $\overline x=\pi-x$, $\overline \rho= \pi-\rho$ and we reduce in studying the positivity of
 	\begin{equation*}
 	\begin{split}
 	G(\overline x,\overline \rho)&= \dfrac{1}{2\pi}\sum_{m=1}^\infty\dfrac{\phi_m}{m^3}\sin(m\overline x)\sin(m\overline \rho)\qquad\forall \overline x,\overline \rho\in\bigg(0,\dfrac{\pi}{4}\bigg),
 	\end{split}
 	\end{equation*}
 	which is already guaranteed by \eqref{eq9}.

 \end{proof}
  
It remains to study the sign of the Green function for $p$ and $q$ near to opposite hinged edges, i.e. near $x=0$ and $\rho=\pi$ or near $x=\pi$ and $\rho=0$; as for Proposition \ref{prop2} the proof is based on a multi step procedure.
At first we prove the following:
\begin{lemma}\label{lemma3}
	Fixed $N\geq 3$, odd integer, $G_p(q)> 0$ if
	\begin{equation*}
	q=(x,y)\in \bigg[\dfrac{\pi}{N+3},\dfrac{\pi}{N+1}\bigg)\times[-\ell,\ell]\, \wedge \, p=(\rho,w)\in \bigg(\pi-\dfrac{\pi}{N+1},\pi\bigg)\times[-\ell,\ell]\,.
	\end{equation*}
\end{lemma}
\begin{proof}
We set $\overline \rho=\pi-\rho$ with $\overline \rho\in\big(0,\frac{\pi}{N+1}\big)$ and we exploit \eqref{goniom}.
For $N\geq3$, odd integer, we rewrite the Green function as
\begin{equation*}
\begin{split}
G(x,\overline \rho)&=\dfrac{1}{2\pi}\sum_{m=1}^\infty(-1)^{m+1}\dfrac{\phi_m}{m^3}\sin(mx)\sin(m\overline \rho)\\&=\dfrac{1}{2\pi}\sum_{\substack{
		m=1\\
		\text{odd}}}^N\bigg[ \dfrac{\phi_m}{m^3}\sin(mx)\sin(m\overline \rho)-\dfrac{\phi_{m+1}}{(m+1)^3}\sin[(m+1)x]\sin[(m+1)\overline \rho]\bigg]\\&+\dfrac{1}{2\pi}\sum_{m=N+2}^\infty (-1)^{m+1}\dfrac{\phi_m}{m^3}\sin(mx)\sin(m\overline\rho)\qquad\forall x,\overline{\rho}\in\bigg(0,\frac{\pi}{N+1}\bigg).
\end{split}
\end{equation*}
By Theorem \ref{monotonia} we know that $\phi_m>0$ and is strictly decreasing with respect to $m\in\mathbb{N}^+$ for all $y,w\in[-\ell,\ell]$; hence we have
\begin{equation}\label{phi1}
\begin{split}
\phi_1 \sin(x)\sin(\overline\rho)&-\dfrac{\phi_2}{2^3}\sin(2x)\sin(2\overline\rho)=\sin(x)\sin(\overline\rho)\bigg[\phi_1 -\dfrac{\phi_2}{2}\cos(x)\cos(\overline\rho)\bigg]\\&>\dfrac{\phi_2}{2}\sin(x)\sin(\overline\rho)>\dfrac{ \phi_{N+1}}{2}\sin(x)\sin(\overline\rho)\qquad \forall x,\overline\rho \in\bigg(0,\frac{\pi}{2}\bigg),\quad \forall N\geq 3.
\end{split}
\end{equation}
Exploiting the inequality 
\begin{equation*}\label{s}
\begin{split}
&\dfrac{\sin(mx)\sin(m\overline\rho)}{m^3}-\dfrac{\sin[(m+1)x]\sin[(m+1)\overline\rho]}{(m+1)^{3}}>\sin(x)\sin(\overline\rho)\bigg[\dfrac{1}{m^{\frac{3}{2}}}-\dfrac{1}{(m+1)^{\frac{3}{2}}}\bigg]^2\\&\hspace{70mm} \forall x,\overline\rho\in\bigg(0,\dfrac{\pi}{N+1}\bigg), \,\,\forall m=3,\dots,N,\quad \forall N\geq 3,\text{odd},
\end{split}
\end{equation*}
(see Lemma \ref{lemmasin2} in the Appendix for a proof) and \eq{phi1}, we get 
\begin{equation}\label{dise4}
\begin{split}
\sum_{\substack{
		m=1\\
		\text{odd}}}^N&\bigg[ \dfrac{\phi_m}{m^3}\sin(mx)\sin(m\overline\rho)-\dfrac{\phi_{m+1}}{(m+1)^3}\sin[(m+1)x]\sin[(m+1)\overline\rho]\bigg]\\&> \phi_{N+1}\sin(x)\sin(\overline\rho) \bigg[\dfrac{1}{2}+\sum_{\substack{
		m=3\\
		\text{odd}}}^N \bigg(\dfrac{1}{m^\frac{3}{2}}-\dfrac{1}{(m+1)^\frac{3}{2}}\bigg)^2\bigg]\quad \forall x,\overline\rho\in\bigg(0,\dfrac{\pi}{N+1}\bigg)\quad \forall N\geq 3,\text{odd}\,.
\end{split}
\end{equation}
On the other hand, since $|\sin(m\overline \rho)|<m\sin(\overline \rho)$ for all $\overline \rho\in(0,\pi)$ and through the monotonicity of the $\phi_m$, we get
\begin{equation}\label{dise5}
\bigg|\sum_{m=N+2}^\infty  (-1)^{m+1}\dfrac{\phi_m}{m^3}\sin(mx)\sin(m\overline\rho)\bigg|< \phi_{N+1}\sin(\overline\rho)\sum_{m=N+2}^\infty  \dfrac{1}{m^2}\quad \forall \overline \rho\in(0,\pi),\,\,\forall N\geq 3.
\end{equation}
From \eqref{dise4}-\eqref{dise5}, for all $N\geq 3$ odd, we infer 
\begin{equation*}\label{aim2}
\begin{split}
G(x,\overline \rho)>
\dfrac{\phi_{N+1}}{2\pi}\sin(\overline\rho)\,\bigg[\dfrac{1}{2}+\sum_{\substack{
		m=3\\
		\text{odd}}}^N &\bigg(\dfrac{1}{m^\frac{3}{2}}-\dfrac{1}{(m+1)^\frac{3}{2}}\bigg)^2\bigg]\,( \sin(x)-\overline C_N) \quad \forall x,\overline\rho\in\bigg(0,\dfrac{\pi}{N+1}\bigg)\,,
\end{split}
\end{equation*}
where
\begin{equation}\label{CCNdef}
\overline C_N:=\dfrac{\sum\limits_{m=N+2}^\infty  \dfrac{1}{m^2}}{\dfrac{1}{2}+\sum\limits_{\substack{
			m=3\\
			\text{odd}}}^N \bigg[\dfrac{1}{m^\frac{3}{2}}-\dfrac{1}{(m+1)^\frac{3}{2}}\bigg]^2} \,.
\end{equation}
Next we denote by $\overline x_N$ the unique solution to the equation
\begin{equation*}\label{sin_ineq2}
\sin(\overline x_N)=\overline C_N\qquad \overline x_N\in(0,\pi/2).
\end{equation*}
The above definition makes sense for all $N\geq 3$, odd, since the map $N\mapsto \overline C_N$ is positive, strictly decreasing and $0<\overline C_N<1$. We prove that
	\begin{equation}\label{ts2}
\overline C_N< \sin\bigg(\dfrac{\pi}{N+3}\bigg)\qquad \forall N\geq 3,\, \text{odd}.
\end{equation}

To this aim we note that
$$
\dfrac{1}{2}+\sum\limits_{\substack{
		m=3\\
		\text{odd}}}^N \bigg[\dfrac{1}{m^\frac{3}{2}}-\dfrac{1}{(m+1)^\frac{3}{2}}\bigg]^2> \dfrac{1}{2}\quad\text{and}\quad\sum\limits_{m=N+2}^\infty  \dfrac{1}{m^2}< \int_{N+1}^\infty\dfrac{1}{x^2}\,dx=\dfrac{1}{N+1},
$$
implying
\begin{equation*}\label{CCN}
\overline C_N< \dfrac{2}{N+1}\qquad \forall N\geq 3.
\end{equation*}
Recalling \eq{stimaseno}, \eq{ts2} follows by noticing that $\dfrac{2}{N+1}\leq \dfrac{3}{N+3}\leq  \sin\bigg(\dfrac{\pi}{N+3}\bigg)$ for all $N\geq 3$.

\end{proof}

In the next Lemma we extend the validity of Lemma \ref{lemma3} to all $\rho\in(3\pi/4,\pi)$.  
\begin{lemma}\label{lemma4}
	Fixed $N\geq 3$, odd integer, $G_p(q)> 0$ if
	\begin{equation*}
	q=(x,y)\in \bigg[\dfrac{\pi}{N+3},\dfrac{\pi}{N+1}\bigg)\times[-\ell,\ell]\, \wedge \, p=(\rho,w)\in \bigg(\dfrac{3}{4}\pi,\pi\bigg)\times[-\ell,\ell]\,.
	\end{equation*}
\end{lemma}
\begin{proof}
	The case $N=3$ is included in the statement of Lemma \ref{lemma3}.\par 
	When $N\geq5$, odd, as explained in the proof of Lemma \ref{lemma3}, we set $\overline \rho=\pi-\rho$ and we exploit \eqref{goniom}, getting through Lemma \ref{lemma3}  that
	\begin{equation}\label{equ1}
	G(x,\overline \rho)>0\qquad\forall x\in\bigg[\dfrac{\pi}{N+3},\dfrac{\pi}{N+1}\bigg),\,\,\forall \overline \rho\in\bigg(0,\dfrac{\pi}{N+1}\bigg)
	\end{equation}
	with $\frac{\pi}{N+1}<\frac{\pi}{4}$.  Moreover, with $\overline C_{N}$ as in \eqref{CCNdef}, we estimate
	\begin{equation}\label{inequ3}
	\begin{split}
	G(x,\overline \rho)&=\dfrac{1}{2\pi}\sum_{\substack{ m=1\\ \text{odd}}}^{N-2}\bigg[ \dfrac{\phi_m}{m^3}\sin(mx)\sin(m\overline \rho)-\dfrac{\phi_{m+1}}{(m+1)^3}\sin[(m+1)x]\sin[(m+1)\overline\rho]\bigg]\\
	&+\dfrac{1}{2\pi}\sum_{m=N}^\infty (-1)^{m+1}\dfrac{\phi_m}{m^3}\sin(mx)\sin(m\overline\rho)\\
	&>\dfrac{\phi_{N-1}}{2\pi}\sin(x)\, \bigg[\dfrac{1}{2}+\sum_{\substack{ m=3 \\ \text{odd}}}^{N-2} \bigg(\dfrac{1}{m^\frac{3}{2}}-\dfrac{1}{(m+1)^\frac{3}{2}}\bigg)^2\bigg]   \bigg\{\sin(\overline \rho)- \overline C_{N-2} \bigg\}\hspace{0mm}\qquad \forall x,\overline\rho\in\bigg(0,\frac{\pi}{N-1}\bigg),
	\end{split}
	\end{equation}
	in which we used Lemma \ref{lemmasin2} of the Appendix with $N-2$ (instead of N) and $|\sin(mx)|< m\sin(x)$ for all $x\in(0,\pi)$. For $N\geq 5$, in the following we denote by $\overline \rho_{ N}$ the unique solution to the equation $\sin(\overline \rho_{ N})=\overline C_{N-2}$, with $\overline C_N$ as in \eqref{CCNdef}. Clearly, $\overline \rho_{ N}=\overline x_{ N}$ and by \eqref{ts2} we obtain
\begin{equation}\label{overlinerho}
	\overline\rho_N<\frac{\pi}{N+1}\,.
	\end{equation}
	Hence, through \eqref{inequ3}, for all $N\geq 5$ odd, we have 
	\begin{equation*}\label{equ3bis}
	G_p(q)>0 \qquad \forall q=(x,y)\in\bigg(0,\frac{\pi}{N-1}\bigg)\times[-\ell,\ell]\, \wedge \,  p=(\overline \rho,w)\in \bigg(\overline \rho_N,\frac{\pi}{N-1} \bigg)\times[-\ell,\ell],
	\end{equation*}
	 implying by \eqref{overlinerho} that
	\begin{equation}\label{equ3}
	G_p(q)>0\qquad\forall q=(x,y)\in\bigg(0,\dfrac{\pi}{N-1}\bigg)\times[-\ell,\ell]\, \wedge \, p=(\overline \rho,w)\in\bigg[\dfrac{\pi}{N+1},\dfrac{\pi}{N-1}\bigg)\times[-\ell,\ell].
	\end{equation}
	Combining \eqref{equ1}-\eqref{equ3}, since $[\frac{\pi}{  N+3},\frac{\pi}{ N+1}) \subset (0,\frac{\pi}{N-1})$, we get
	\begin{equation}\label{equ1bis}
	G(x,\overline \rho)>0\qquad\forall x\in\bigg[\dfrac{\pi}{N+3},\dfrac{\pi}{N+1}\bigg),\,\,\forall \overline \rho\in\bigg(0,\dfrac{\pi}{N-1}\bigg)\,.
	\end{equation}
	If $N=5$ the above inequality yields the thesis since $\dfrac{\pi}{N-1}=\dfrac{\pi}{4}$.\par Fix now $\overline N \geq 7$, odd, from \eqref{equ1bis} we clearly have
	$$
	G(x,\overline \rho)>0\qquad\forall x\in\bigg[\dfrac{\pi}{ \overline N+3},\dfrac{\pi}{ \overline N+1}\bigg),\,\,\forall \overline \rho\in\bigg(0,\dfrac{\pi}{ \overline N-1}\bigg)\,;
	$$
	moreover, since $[\frac{\pi}{ \overline { N}+3},\frac{\pi}{ \overline { N}+1}) \subset (0,\frac{\pi}{N-1})$  for all $N\leq \overline N-2$, through \eqref{equ3} we also have
\begin{equation*}
\begin{split}
G(x,\overline \rho)>0\qquad\forall x\in\bigg[\dfrac{\pi}{ \overline N+3},\dfrac{\pi}{ \overline N+1}\bigg),\,\,\forall\overline \rho\in\bigg[\dfrac{\pi}{N+1},\dfrac{\pi}{N-1}\bigg)\qquad \forall N= 5,..., \overline { N}-2, \text{ odd}
\end{split}
\end{equation*}
and the thesis follows by noticing that $ \bigg(0,\dfrac{\pi}{ \overline { N}-1}\bigg) \bigcup \left( \displaystyle{\bigcup_{\substack{ N=5 \\ \text{odd}}}^{\overline N-2}} \bigg[\dfrac{\pi}{N+1},\dfrac{\pi}{N-1}\bigg)\right)= \bigg(0,\dfrac{\pi}{4}\bigg)$.

\end{proof}

By Lemma \ref{lemma4} we derive the following:
 \begin{proposition}\label{prop3}
	$G_p(q)> 0$ if
	\begin{equation*}
	q=(x,y)\in \bigg(0,\dfrac{\pi}{4}\bigg)\times[-\ell,\ell]\, \wedge \, p=(\rho,w)\in \bigg(\dfrac{3}{4}\pi,\pi\bigg)\times[-\ell,\ell]\,
	\end{equation*}
	\begin{center}
		or
	\end{center}
	\begin{equation*}
	q=(x,y)\in \bigg(\dfrac{3}{4}\pi,\pi\bigg)\times[-\ell,\ell]\, \wedge \, p=(\rho,w)\in \bigg(0,\dfrac{\pi}{4}\bigg)\times[-\ell,\ell]\,.
	\end{equation*}
\end{proposition}

\begin{proof}
	From Lemma \ref{lemma4} we have 
	\begin{equation*}\label{eq18}
	G_p(q)>0\qquad\forall q=(x,y)\in\bigg[\dfrac{\pi}{N+3},\dfrac{\pi}{4}\bigg)\times[-\ell,\ell]\, \wedge \, p=(\rho,w)\in\bigg(\dfrac{3}{4}\pi,\pi\bigg)\times[-\ell,\ell]\quad \forall N\geq 3,\text{odd}.
	\end{equation*}
	Passing to the limit  $N\rightarrow+\infty$ we obtain 
	\begin{equation}\label{eq19}
	G_p(q)>0\qquad\forall q=(x,y)\in\bigg(0,\dfrac{\pi}{4}\bigg)\times[-\ell,\ell]\, \wedge \, p=(\rho,w)\in\bigg(\dfrac{3}{4}\pi,\pi\bigg)\times[-\ell,\ell].
	\end{equation}
	
	To complete the proof we exploit the trigonometric relations \eqref{goniom}.
	Then, for $x\in\big(\frac{3}{4}\pi,\pi\big)$ and $\rho\in\big(0,\frac{\pi}{4}\pi\big)$ we set $\overline x=\pi-x$, $\overline \rho=\pi-\rho$ and we study the positivity of
	\begin{equation*}
	\begin{split}
	G(\overline x,\overline \rho)&= \dfrac{1}{2\pi}\sum_{m=1}^\infty(-1)^{m+1}\dfrac{\phi_m}{m^3}\sin(m\overline x)\sin(m\overline \rho)\quad\forall\overline x\in\bigg(0,\dfrac{\pi}{4}\bigg)\,\,\forall\overline \rho\in\bigg(\dfrac{3}{4}\pi,\pi\bigg),
	\end{split}
	\end{equation*}
	already guaranteed by \eqref{eq19}.	
\end{proof}

\textbf{Proof of Theorem \ref{greenpositiva} completed}. The proof follows by combining the statements of Propositions \ref{prop51},  \ref{prop2} and  \ref{prop3}.

\section{Appendix}

\subsection{Appendix A}
In this section we collect two technical results used in Section \ref{proof1} to prove the monotonicity issue \eqref{decreasing} of Theorem \ref{monotonia}.
\begin{lemma}\label{AF_F1}
	Given $F(z)$ and $\overline F(z)$ as in \eq{cost1}, there holds:
	$$\dfrac{2F(z)-F'(z)}{[F(z)]^2}\pm\bigg[\dfrac{2\overline F(z)-\overline F'(z)}{[\overline F(z)]^2}\bigg]<0\qquad \forall z>0.$$
\end{lemma}
\begin{proof}
	We have
	\begin{equation*}
	\begin{split}
	2F(z)-F'(z)&=(3+\sigma)[\sinh(2z)-\cosh(2z)]+(1-\sigma)(1-2z):=\alpha(z)\\2\overline F(z)-\overline F'(z)&=(3+\sigma)[\sinh(2z)-\cosh(2z)]-(1-\sigma)(1-2z)
	\end{split}
	\end{equation*}
	and we observe that $\alpha(0)=-2(1+\sigma)<0$ and $\alpha(z)\sim -2z(1-\sigma)\rightarrow-\infty$ for $z\rightarrow +\infty$; $\alpha(z)$ has stationary points satisfying $(3+\sigma)[\sinh(2\overline z)-\cosh(2\overline z)]=-(1-\sigma)$, hence $\alpha(\overline z)=-2(1-\sigma)\overline z<0$ for all $\overline z>0$, implying $\alpha(z)<0$ for all $z>0$. \par 
	From above and recalling that $\overline F(z)>F(z)>0$ for all $z>0$, a sufficient condition for the thesis is: $2F(z)-F'(z)\pm[2\overline F(z)-\overline F'(z)]<0$ for all $z>0$. To this aim we notice that:
	\begin{equation*}
	\begin{split}
	&2F(z)-F'(z)+2\overline F(z)-\overline F'(z)=2(3+\sigma)[\sinh(2z)-\cosh(2z)]<0\qquad\forall z> 0\\&2F(z)-F'(z)-2\overline F(z)+\overline F'(z)=2(1-\sigma)(1-2z)<0\qquad\hspace{22mm} \forall z>\frac{1}{2}.
	\end{split}
	\end{equation*}
	Therefore, it remains to prove the negativity of the term below for all $0<z\leq1/2$:
	\begin{equation*}
	\begin{split}
	&\dfrac{2F(z)-F'(z)}{[F(z)]^2}-\dfrac{2\overline F(z)-\overline F'(z)}{[\overline F(z)]^2}\\&=\dfrac{2(1-\sigma)}{[F(z)\overline F(z)]^2}\bigg\{\dfrac{(3+\sigma)^2}{16}\bigg[e^{4z}\big(1-2z\big)+e^{-4z}(1+6z)-4z-2\bigg]+(1-\sigma)^2(z^2-2z^3)\bigg\}\\&:=\dfrac{2(1-\sigma)}{[F(z)\overline F(z)]^2}\beta(z).
	\end{split}
	\end{equation*}
	We want to prove that $\beta(z)<0$ for all $0<z\leq1/2$; we compute
	\begin{equation*}
	\begin{split}
&\beta'(z)=\dfrac{(3+\sigma)^2}{8}\bigg[e^{4z}\big(1-4z\big)+e^{-4z}(1-12z)-2\bigg]+2(1-\sigma)^2(z-3z^2)\\&\beta''(z)=2\{(3+\sigma)^2\big[-ze^{4z}+e^{-4z}(3z-1)\big]+(1-\sigma)^2(1-6z)\}.
	\end{split}
	\end{equation*} 
We observe that $\beta(0)=\beta'(0)=0$ and $\beta''(0)<0$, hence if we prove $\beta''(z)<0$ in $(0,\frac{1}{2}]$ we have the thesis;
we note that
$$\max_{z\in[\frac{1}{3},\frac{1}{2}]}-ze^{4z}=-\frac{e^{4/3}}{3}\qquad\max_{z\in[\frac{1}{3},\frac{1}{2}]}e^{-4z}(3z-1)=\frac{1}{2e^2},$$
so that $\frac{1}{2e^2}-\frac{e^{4/3}}{3}<0$ and, in turn, $\beta''(z)<0$ for $z\in[\frac{1}{3},\frac{1}{2}]$. Moreover $\beta''(z)<0$ also for $z\in[\frac{1}{6},\frac{1}{3}]$, since is given by the sum of non positive terms; this implies $\beta''(z)<0$ for all $z\in[\frac{1}{6},\frac{1}{2}]$. It remains to establish the sign of $\beta''(z)$ in $(0,\frac{1}{6}]$; for $z\in (0,\frac{1}{6}]$ we observe that the unique positive term in $\beta''(z)$ is $(1-6z)$, therefore if we prove $\overline \beta(z):=-ze^{4z}+e^{-4z}(3z-1)+1-6z<0$ for all $z\in  (0,\frac{1}{6}]$
we have $\beta''(z)<0$ for $z\in \big(0,\frac{1}{6}\big]$. We get
$
\overline \beta'(z)=-6-e^{4z}(1+4z)+e^{-4z}(7-12z)$ and $
\overline \beta''(z)=-8\big[e^{4z}(1+2z)+e^{-4z}(5-6z)\big]<0$  for  $z\in\big(0,\frac{1}{6}\big]$; being $\overline \beta(0)=\overline\beta'(0)=0$, we have $\overline \beta(z)<0$ for all $z\in \big(0,\frac{1}{6}\big]$ and the thesis.

\end{proof}

\begin{lemma}\label{parita}
	Given $\zeta(k,z)$, $\eta(k,z)$, $\psi(k,z)$, $\xi(k,z)$ as in \eq{cost}, we have that
	\begin{align*}
	\zeta(-k,z)&=\zeta(k,z),\quad &\zeta_z(-k,z)&=\zeta_z(k,z),\quad &\forall k\in[-1,1], \forall z>0,\\
	\eta(-k,z)&=-\eta(k,z),\quad   &\eta_z(-k,z)&=-\eta_z(k,z),\quad &\forall k\in[-1,1], \forall z>0,\\
\psi(-k,z)&=\psi(k,z),\quad& \psi_z(-k,z)&=\psi_z(k,z),\quad &\forall k\in[-1,1], \forall z>0,\\
\xi(-k,z)&=-\xi(k,z),\quad &\xi_z(-k,z)&=-\xi_z(k,z)\quad &\forall k\in[-1,1]\,\forall z>0\,.
\end{align*}
	
\end{lemma}
\begin{proof}
The proof follows by a direct inspection of the analytic expression of each function that we write here below. 
\begin{equation*}
\begin{split}
\zeta(k,z)=&\bigg(\dfrac{4}{1-\sigma}-z(1+\sigma)\bigg)\cosh(kz)\cosh(z)+\bigg(\dfrac{(1+\sigma)^2}{1-\sigma}+2z\bigg)\cosh(kz)\sinh(z)\\&-2kz\sinh(kz)\cosh(z)+kz(1+\sigma)\sinh(kz)\sinh(z)\\
\zeta_z(k,z)=&\bigg(\dfrac{2\sigma(1+\sigma)}{1-\sigma}+2z(1-k^2)\bigg)\cosh(kz)\cosh(z)+\bigg(\dfrac{2(3-\sigma)}{1-\sigma}-z(1+\sigma)(1-k^2)\bigg)\cosh(kz)\sinh(z)\\&+k\dfrac{2(1+\sigma)}{1-\sigma}\sinh(kz)\cosh(z)+k\dfrac{2(1+\sigma)}{1-\sigma}\sinh(kz)\sinh(z),
\end{split}
\end{equation*}
\begin{equation*}
\begin{split}
\eta(k,z)=&kz(1+\sigma)\cosh(kz)\cosh(z)-2kz\cosh(kz)\sinh(z)\\&+\sinh(kz)\cosh(z)\bigg(\dfrac{(1+\sigma)^2}{1-\sigma}+2z\bigg)+\sinh(kz)\sinh(z)\bigg(\dfrac{4}{1-\sigma}-z(1+\sigma)\bigg)\\
\eta_z(k,z)=&k\dfrac{2(1+\sigma)}{1-\sigma}\cosh(kz)\cosh(z)+k\dfrac{2(1+\sigma)}{1-\sigma}\cosh(kz)\sinh(z)\\&+\bigg(\dfrac{2(3-\sigma)}{1-\sigma}-z(1+\sigma)(1-k^2)\bigg)\sinh(kz)\cosh(z)+\bigg(\dfrac{2\sigma(1+\sigma)}{1-\sigma}+2z(1-k^2)\bigg)\sinh(kz)\sinh(z),
\end{split}
\end{equation*}
\begin{equation*}
\begin{split}
\psi(k,z)=&\big[2+z(1-\sigma)\big]\cosh(kz)\cosh(z)+\big[-(1+\sigma)+z(1-\sigma)\big]\cosh(kz)\sinh(z)\\&-kz(1-\sigma)\sinh(kz)\cosh(z)-kz(1-\sigma)\sinh(kz)\sinh(z)\\
\psi_z(k,z)=&\big[-2\sigma+z(1-k^2)(1-\sigma)\big]\cosh(kz)\cosh(z)+\big[3-\sigma+z(1-k^2)(1-\sigma)\big]\cosh(kz)\sinh(z)\\&+k(1+\sigma)\sinh(kz)\cosh(z)-2k\sinh(kz)\sinh(z)
\end{split}
\end{equation*}
\begin{equation*}
\begin{split}
\xi(k,z)=&-kz(1-\sigma)\cosh(kz)\cosh(z)-kz(1-\sigma)\cosh(kz)\sinh(z)\\&+\sinh(kz)\cosh(z)\big[-(1+\sigma)+z(1-\sigma)\big]+\sinh(kz)\sinh(z)\big[2+z(1-\sigma)\big]\\
\xi_z(k,z)=&-2k\cosh(kz)\cosh(z)+k(1+\sigma)\cosh(kz)\sinh(z)\\&+\big[3-\sigma+z(1-\sigma)(1-k^2)\big]\sinh(kz)\cosh(z)+\big[-2\sigma+z(1-\sigma)(1-k^2)\big]\sinh(kz)\sinh(z).
\end{split}
\end{equation*}
Clearly, the explicit form of the derivatives is not needed to prove the statement but it is used in the proofs of Section \ref{proof1}, this is the reason why we decided to write it down here.

\end{proof}

\subsection{Appendix B}
In this section we prove two inequalities that we repeatedly exploit to estimate the Fourier series of the Green function in Section \ref{proof2}.
\begin{lemma}\label{lemmasin}
 	Let $N\geq 2$ be an integer. There holds
 	\begin{equation}\label{sin2}
 	\sin(m\rho)\sin(mx)>\sin(\rho)\sin(x)\qquad \forall \rho,x\in\bigg(0,\dfrac{\pi}{N+1}\bigg),\quad \forall m=2,\dots,N.
 	\end{equation}
 \end{lemma}

\begin{proof} 
	Let $m\geq 2$ be an integer. Using the complex identity $\sin(mx)=\frac{1}{2i}(e^{imx}-e^{-imx})$ it is readily seen that 
	$$
	\sin(mx)-\sin(x)=0\quad \iff \quad  x=\dfrac{2k\pi}{m-1}\lor  x=\dfrac{(2k+1)\pi}{m+1}\qquad \text{for}\quad k\in\mathbb{Z}.
	$$
	By this, since $(\sin(mx)-\sin(x))'(0)=m-1>0$, we infer that
	$$ 
	\sin(mx)-\sin(x)>0\qquad \forall x\in\bigg(0,\dfrac{\pi}{m+1}\bigg)
	$$
	and, in turn, that
	\begin{equation*}\label{sin}
	\sin(mx)>\sin(x)\qquad \forall x\in\bigg(0,\dfrac{\pi}{N+1}\bigg),\quad \forall m=2,\dots,N.
	\end{equation*} 
Finally, we obtain \eq{sin2} thanks to the elementary implication
	\begin{equation*}\label{dis4}
	\forall a,b,c,d\in\mathbb{R}^+\,,\quad (a-b>0 \wedge c-d>0)\quad \Rightarrow\quad ac>bd\,.
	\end{equation*}
	
\end{proof}

 Before stating and prove the second inequality we need the following lemma.

\begin{lemma}\label{zeri}
Let $m\geq 3$ be an integer and set $a_m:=\big[\frac{1}{m^{3/2}}-\frac{1}{(m+1)^{3/2}}\big]^2$. The function
	$$\upsilon_m(t):=\frac{\sin(m t)}{m^2}-\frac{\sin[(m+1) t]}{(m+1)^{2}}-a_m \,\sin(t)\,, \quad\quad t\in [-\frac{2\pi}{m+1},\frac{2\pi}{m+1}\big]$$
	vanishes at $t=0$ and $t=\pm t_1$ with $t_1\in\big(\frac{2\pi}{2m+1},\frac{3\pi}{2m}\big)$. Furthermore, $\upsilon_m(t)>0$ in $[-\frac{2\pi}{m+1},-t_1)$ and  $(0,t_1)$ while $\upsilon_m(t)<0$ in $(-t_1,0)$ and  $(t_1,\frac{2\pi}{m+1}\big]$.
\end{lemma}

\begin{proof}
Since $\upsilon_m$ is odd it is sufficient to study its behaviour in $[0,\frac{2\pi}{m+1}\big]$. Clearly, $\upsilon_m(0)=0$. We compute 
	$$
	\upsilon_m'(t)=\frac{\cos(m t)}{m}-\frac{\cos[(m+1) t]}{m+1}-a_m\cos(t)\qquad\upsilon_m''(t)=\sin[(m+1)t]-\sin(mt)+a_m\sin(t).
	$$
	Since $\frac{2\pi}{m+1}\leq \frac{\pi}{2}$ for all $m\geq 3$, then $\cos(t)$ (and $\sin(t)$) are always non negative.
	We observe that, for all $m\geq 3$, $\upsilon_m'(0)=\frac{m^2-m-1}{m^3(m+1)}+\frac{1}{(m+1)^{3/2}}\big(\frac{2}{m^{3/2}}-\frac{1}{(m+1)^{3/2}}\big)>0$ and $$\upsilon_m\bigg(\frac{2\pi}{m+1}\bigg)=-\sin\bigg(\frac{2\pi}{m+1}\bigg)\bigg[\frac{1}{m^2}+a_m\bigg]<0,$$ therefore there exists at least another zero for $\upsilon_m(t)$ in $\big[0,\frac{2\pi}{m+1}\big]$, we prove that it is unique.\par

	 As in the proof of Lemma \ref{lemmasin} we use the complex identities for the trigonometric functions and we obtain
	\begin{equation*}
	\begin{split}
	\sin[(m+1)\overline t]-\sin(m\overline t)=0\quad &\iff\quad \overline t=2k\pi,\dfrac{(1+2k)\pi}{2m+1}\quad \forall k\in\mathbb{Z}\\	\cos[(m+1)\overline t]-\cos(m\overline t)=0\quad &\iff\quad \overline t=2k\pi,\dfrac{2k\pi}{2m+1}\quad\hspace{5mm}\forall k\in\mathbb{Z}.
	\end{split}
	\end{equation*}
	Hence $\sin[(m+1) t]>\sin(m t)$ for $t\in\big(0,\frac{\pi}{2m+1}\big)\cup\big(\frac{3\pi}{2m+1},\frac{2\pi}{m+1}\big]$ and $\upsilon_m''(t)>0$ for $t\in\big(0,\frac{\pi}{2m+1}\big)\cup\big(\frac{3\pi}{2m+1},\frac{2\pi}{m+1}\big]$; this readily implies that $\upsilon_m(t)>0$ for $t\in\big(0,\frac{\pi}{2m+1}\big]$. Moreover we have
	$$
	\upsilon_m\bigg(\dfrac{3\pi}{2m}\bigg)=-\dfrac{1}{m^2}-\dfrac{1}{(m+1)^2}\sin\bigg[(m+1)\dfrac{3\pi}{2m}\bigg]-a_m\sin\bigg(\dfrac{3\pi}{2m}\bigg)<0
	$$
	since $\frac{3\pi}{2m}<\pi$ and $-\frac{1}{m^2}+\frac{1}{(m+1)^2}<0$ for all $m\geq 3$; this information combined with $\upsilon_m\big(\frac{2\pi}{m+1}\big)<0$ and the convexity of $\upsilon_m (t)$, implies $\upsilon_m(t)<0$ for $t\in\big[\frac{3\pi}{2m},\frac{2\pi}{m+1}\big]$. \par 
	
	On the other hand $\cos(mt)<\cos[(m+1)t]$ for $t\in\big(\frac{2\pi}{2m+1},\frac{2\pi}{m+1}\big]$; since $\frac{m}{m+1}<1$ and $\cos(mt)< 0$ for $t\in\big(\frac{2\pi}{2m+1},\frac{3\pi}{2m}\big)$, we have $\cos(mt)<\frac{m}{m+1}\cos[(m+1)t]$ for $t\in\big(\frac{2\pi}{2m+1},\frac{3\pi}{2m}\big)$, implying $\upsilon_m'(t)<0$ for $t\in\big(\frac{2\pi}{2m+1},\frac{3\pi}{2m}\big)$ and the possibility of a unique zero for $\upsilon_m (t)$ in $\big(\frac{2\pi}{2m+1},\frac{3\pi}{2m}\big)$. Therefore if $\upsilon_m (t)$ is positive in $\big(\frac{\pi}{2m+1},\frac{2\pi}{2m+1}\big]$ we conclude that there are not other zeroes.\par
	
	We consider
	$$
	\upsilon_m(t)>\frac{\sin(m t)}{m^2}-\frac{\sin[(m+1) t]}{(m+1)^{2}}-\dfrac{1}{m^3}+\dfrac{1}{(m+1)^{\frac{3}{2}}}\bigg[\dfrac{2}{m^{\frac{3}{2}}}-\dfrac{1}{(m+1)^{\frac{3}{2}}}\bigg]\sin(t),
	$$
	 and we study the sign of $\overline \upsilon_m (t)=\frac{\sin[m t]}{m^2}-\frac{\sin[(m+1) t]}{(m+1)^{2}}-\frac{1}{m^3}$ for $t\in\big(\frac{\pi}{2m+1},\frac{2\pi}{2m+1}\big]$. We have $$\overline \upsilon_m'' (t)=\sin[(m+1) t]-\sin(m t)<0\qquad \forall t\in\bigg(\frac{\pi}{2m+1},\frac{2\pi}{2m+1}\bigg], $$
	thus if $\overline \upsilon_m \big(\frac{\pi}{2m+1}\big)>0$ and $\overline \upsilon_m \big(\frac{2\pi}{2m+1}\big)>0$ we conclude that $\overline \upsilon_m(t)>0$ for $t\in\big(\frac{\pi}{2m+1},\frac{2\pi}{2m+1}\big]$ and, in turn, $\upsilon_m(t)>0$ for all $t\in\big(\frac{\pi}{2m+1},\frac{2\pi}{2m+1}\big]$.\par 
	Recalling that $\sin\big(\frac{m\pi}{2m+1}\big)=\sin\big(\frac{(m+1)\pi}{2m+1}\big)$ and that,  for all $m\geq 3$, $\frac{m \pi}{2m+1}\geq \frac{3\pi}{7}>\frac{\pi}{3}$, we estimate
	$$
	\overline \upsilon_m \bigg(\frac{\pi}{2m+1}\bigg)=\sin\bigg(\dfrac{m\pi}{2m+1}\bigg)\bigg[\dfrac{1}{m^2}-\dfrac{1}{(m+1)^2}\bigg]-\dfrac{1}{m^3}>\dfrac{\sqrt{3}}{2}\bigg[\dfrac{1}{m^2}-\dfrac{1}{(m+1)^2}\bigg]-\dfrac{1}{m^3}
	$$
	where the last term is positive for $m>\frac{1}{8}(1+\sqrt{3})\big[4-\sqrt{3}+\sqrt{3+8\sqrt{3}}\big]\approx 2.18$.
	
	Finally, exploiting the fact that $\sin\big(\frac{2(m+1)\pi}{2m+1}\big)=-\sin\big(\frac{2m\pi}{2m+1}\big)=-\sin\big(\frac{\pi}{2m+1}\big)$ and \eq{stimaseno}, for all $m\geq 3$, we infer
$$
\overline \upsilon_m \bigg(\frac{2\pi}{2m+1}\bigg)=\sin\bigg(\dfrac{\pi}{2m+1}\bigg)\bigg[\dfrac{1}{m^2}+\dfrac{1}{(m+1)^2}\bigg]-\dfrac{1}{m^3}>\dfrac{(m+1)(2m-\sqrt{3}-1)(2m+\sqrt{3}-1)}{2m^3(m+1)^3}>0.
$$
This concludes the proof.
\end{proof}

 \begin{lemma}\label{lemmasin2}
	Let $N\geq 3$ be an integer. There holds
	
	\begin{equation}\label{sin3}
	\begin{split}
\dfrac{\sin(m\rho)\sin(mx)}{m^3}&-\dfrac{\sin[(m+1)\rho]\sin[(m+1)x]}{(m+1)^{3}}\\&>\sin(\rho)\sin(x)\bigg[\dfrac{1}{m^{\frac{3}{2}}}-\dfrac{1}{(m+1)^{\frac{3}{2}}}\bigg]^2\qquad \forall \rho,x\in\bigg(0,\dfrac{\pi}{N+1}\bigg), \forall m=3,\dots,N.
	\end{split}
	\end{equation}
\end{lemma}

\begin{proof}
	As in the previous lemma we set $a_m:=\big[\frac{1}{m^{3/2}}-\frac{1}{(m+1)^{3/2}}\big]^2$ and
	we consider the function 
	$$
	\overline S_m(\rho,x)=\dfrac{\sin(m\rho)\sin(mx)}{m^3}-\dfrac{\sin[(m+1)\rho]\sin[(m+1)x]}{(m+1)^{3}}-a_m\sin(\rho)\sin(x)$$
	 with $Q:=\big[0,\frac{\pi}{m+1}\big]^2$. By Weierstrass Theorem $\overline S_m$ admits maximum and minimum in $Q$. To locate the stationary points of $\overline S_m$ it is convenient to exploit the following change of variables:
\begin{equation*}
\begin{cases}
\rho+x=\gamma\\
\rho-x=\iota
\end{cases}
\qquad \Rightarrow\qquad \begin{cases}
\rho=\frac{\gamma+\iota}{2}\\
x=\frac{\gamma-\iota}{2} \,
\end{cases}
\end{equation*}
according to which $\overline S_m$ reads
$$\overline S_m(\iota,\gamma)=\dfrac{1}{2}\bigg(\dfrac{\cos(m\iota)-\cos(m\gamma)}{m^3}-\dfrac{\cos[(m+1)\iota]-\cos[(m+1)\gamma]}{(m+1)^{3}}-a_m\big[\cos(\iota)-\cos(\gamma)\big]\bigg)$$
for $(\iota,\gamma)\in Q_1:=\big[-\frac{\pi}{m+1},\frac{\pi}{m+1}\big]\times\big[0,\frac{2\pi}{m+1}\big]$.
 Let $\upsilon_m$ be as defined in the statement of Lemma \ref{zeri}, we have
	\begin{equation*}
	\begin{cases}
	\frac{\partial \overline S_m}{\partial \iota}(\iota,\gamma)&=-\frac{\upsilon_m(\iota)}{2}=0\\ 
	\frac{\partial \overline S_m}{\partial \gamma}(\iota,\gamma)&=\frac{\upsilon_m(\gamma)}{2}=0\,.
	\end{cases}
	\end{equation*} 
	By Lemma \ref{zeri} it is readily deduced that $\overline S_m$ admits only two stationary points $(0,0)$
  and $(0,t_1)$ with $t_1\in\big(\frac{2\pi}{2m+1},\frac{3\pi}{2m}\big)$. 	 Since  $(0,0)\in \partial Q_1$, we only need to study the nature of $(0,t_1)$. We have $\frac{\partial^2 \overline S_m(0,t_1)}{\partial \iota^2}=-\upsilon_m'(0)<0$, $\frac{\partial^2 \overline S_m}{\partial \iota\partial \gamma}(0,t_1)=0$ and $\frac{\partial^2 \overline S_m(0,t_1)}{\partial \gamma^2}=\upsilon_m'\big(t_1\big)<0$, see the proof of Lemma \ref{zeri}. This implies that $(0,t_1)$ is a maximum point. \par 
  Coming back to the original variables, from the above analysis we infer that $(\frac{t_1}{2},\frac{t_1}{2})$ is a local maximum point for $\overline S_m$. About $\overline S_m(\rho,x)$ constrained to $\partial Q$, we have: $\overline S_m(0,x)=\overline S_m(\rho,0)=0$ for all $x,\rho \in \big[0,\frac{\pi}{m+1}\big]$; hence, we only need to study $\overline S_m(\frac{\pi}{m+1},x)$ and $\overline S_m(\rho,\frac{\pi}{m+1})$  for all $x,\rho \in \big[0,\frac{\pi}{m+1}\big]$, since they have the same analytic expression, we only focus on
   $$\overline S_m\bigg(\frac{\pi}{m+1},x\bigg)=\sin\bigg(\frac{\pi}{m+1}\bigg)\bigg[\dfrac{\sin(mx)}{m^3}-a_m\sin(x)\bigg]:=\sin\bigg(\frac{\pi}{m+1}\bigg)\widetilde{h}_m(x) \quad \text{for } x \in \big[0,\frac{\pi}{m+1}\big]\,,$$
   where $\sin\big(\frac{m\pi}{m+1}\big)=\sin\big(\frac{\pi}{m+1}\big)>0$ for all $m\geq 3$. We have $\widetilde{h}_m(0) =0$ and $$\widetilde{h}_m\bigg(\frac{\pi}{m+1} \bigg)=\sin\bigg(\frac{\pi}{m+1}\bigg)\dfrac{1}{(m+1)^{\frac{3}{2}}}\bigg[\dfrac{2}{m^{\frac{3}{2}}}-\dfrac{1}{(m+1)^{\frac{3}{2}}}\bigg]>0\qquad \forall m\geq 3.
	$$
	We study the sign of the stationary points of $\widetilde h_m(x)$ for $x\in\big[0,\frac{\pi}{m+1}\big]$; we get
	$$
	\widetilde{h}_m'(\hat x)=0\quad\iff\quad a_m\cos(\hat x)=\dfrac{\cos(m\hat x)}{m^2};
	$$	
	since $\cos(\hat x)>0$ for $\hat x\in\big[0,\frac{\pi}{m+1}\big]$ we rewrite $\widetilde{h}(\hat x)$ as follows
	$$
	\widetilde{h}_m(\hat x)=\dfrac{1}{m^2\cos(\hat x)}\bigg(\dfrac{\sin(m\hat x)\cos(\hat x)}{m}-\cos(m\hat x)\sin(\hat x)\bigg):=\dfrac{1}{m^2\cos(\hat x)}\widetilde\eta_m (\hat x).
	$$ 
	We have $\widetilde \eta_m(\hat x)=\big(\frac{m^2-1}{3}\big)\hat x^3+o(\hat x^3)$ for $\hat x\rightarrow 0$ and $\widetilde \eta_m '(\hat x)=(m-1/m)\sin(m\hat x)\sin(\hat x)\geq 0$, being $\sin(m\hat x)\geq0$ for $\hat x\in\big[0,\frac{\pi}{m+1}\big]$. Hence, $\widetilde{h}_m(\hat x)>0$. This, combined with the fact that $\widetilde{h}_m(0 )=0$ and $\widetilde{h}_m\bigg(\frac{\pi}{m+1} \bigg)>0$ implies that $\widetilde h_m(x)\geq0$ for all $ x\in\big[0,\frac{\pi}{m+1}\big]$.\par
	 Once established that $\overline S_m(\rho,x)$ is non-negative on $\partial Q$ and admits no internal minimum points in $Q$, we conclude that $\overline S_m(\rho,x)\geq 0$ for all $(\rho,x)\in Q$. Finally, the strict inequality in \eq{sin3} comes by observing that $ \overline S_m(\rho,x)=0$ for $(\rho,x)\in Q$, if and only if $\rho=0$ or $x=0$.
\end{proof}

\par\noindent
{\small \textbf{Acknowledgments.} The authors are members of the Gruppo Nazionale per l'Analisi Matematica, la Probabilit\`a e le loro Applicazioni (GNAMPA) of the Istituto Nazionale di Alta Matematica (INdAM) and are partially supported by the INDAM-GNAMPA 2019 grant: ``Analisi spettrale per operatori ellittici con condizioni di Steklov o parzialmente incernierate'' and by the PRIN project 201758MTR2: ``Direct and inverse problems for partial differential equations: theoretical aspects and applications'' (Italy). 
}

\end{document}